\documentclass[11pt]{article}
 \usepackage[mathscr]{eucal}
\usepackage{epsfig,epsf,psfrag}
\usepackage{amssymb,amsfonts,latexsym}
\usepackage{amsmath}
\usepackage{graphicx}
\usepackage{epstopdf}
\usepackage{bm,xcolor,url}
\usepackage{fixltx2e}
\usepackage{array}
\usepackage{verbatim}
\usepackage{algorithm}
\usepackage{algpseudocode}
\usepackage{cite}
\usepackage{verbatim}
\usepackage{textcomp}
\usepackage{mathrsfs}
\usepackage[referable]{threeparttablex}
\usepackage{subfig}
\usepackage{amsthm}
\usepackage{enumerate}
\usepackage{footnote}
\usepackage{enumitem}
\usepackage{xr}
\usepackage{mathtools}
\catcode`~=11 \def\UrlSpecials{\do\~{\kern -.15em\lower .7ex\hbox{~}\kern .04em}} \catcode`~=13 

\allowdisplaybreaks[3]

\newcommand{\tnorm}[1]{{\left\vert\kern-0.25ex\left\vert\kern-0.25ex\left\vert #1 
    \right\vert\kern-0.25ex\right\vert\kern-0.25ex\right\vert}}
\newcommand{\tnormt}[1]{{\vert\kern-0.25ex\vert\kern-0.25ex\vert #1 
    \vert\kern-0.25ex\vert\kern-0.25ex\vert}}

\newcommand{\andd}{\mbox{and}}

\newcommand{\where}{\mathrm{where}}

\newcommand{\normt}[1]{\Vert#1\Vert}
\newcommand{\abst}[1]{\vert#1\vert}
\newcommand{\abs}[1]{\left\lvert#1\right\rvert}

\newcommand{\nn}{\nonumber}

\newcommand{\dom}{\mathsf{dom}\,}

\newcommand{\inter}{\mathsf{int}\,}
\newcommand{\bdry}{\mathsf{bd}\,}
\newcommand{\diag}{\mathsf{diag}\,}
\newcommand{\cl}{\mathsf{cl}\,}
\newcommand{\ri}{\mathsf{ri}\,}

\newcommand{\calA}{\mathcal{A}}

\newcommand{\calC}{\mathcal{C}}
\newcommand{\calD}{\mathcal{D}}

\newcommand{\calH}{\mathcal{H}}
\newcommand{\calI}{\mathcal{I}}

\newcommand{\calK}{\mathcal{K}}
\newcommand{\calL}{\mathcal{L}}

\newcommand{\calQ}{\mathcal{Q}}
\newcommand{\calR}{\mathcal{R}}
\newcommand{\calS}{\mathcal{S}}

\newcommand{\calX}{\mathcal{X}}
\newcommand{\calY}{\mathcal{Y}}
\newcommand{\calZ}{\mathcal{Z}}

\newcommand{\rmc}{\mathrm{c}}

\newcommand{\rmd}{\mathrm{d}}

\newcommand{\bbA}{\mathbb{A}}

\newcommand{\bbC}{\mathbb{C}}
\newcommand{\bbD}{\mathbb{D}}

\newcommand{\bbF}{\mathbb{F}}

\newcommand{\bbH}{\mathbb{H}}

\newcommand{\bbL}{\mathbb{L}}

\newcommand{\bbO}{\mathbb{O}}

\newcommand{\bbR}{\mathbb{R}}
\newcommand{\bbS}{\mathbb{S}}

\newcommand{\bbV}{\mathbb{V}}

\newcommand{\bbY}{\mathbb{Y}}
\newcommand{\bbZ}{\mathbb{Z}}

\newcommand{\barbbR}{\bar{\bbR}}

\DeclareMathAlphabet{\mathbsf}{OT1}{cmss}{bx}{n}

\newcommand{\rvA}{\mathsf{A}}

\newcommand{\rvL}{\mathsf{L}}

\newcommand{\hatx}{\hat{x}}
\newcommand{\hatX}{\widehat{X}}

\newcommand{\barf}{\bar{f}}
\newcommand{\barg}{\bar{g}}

\newcommand{\barx}{\bar{x}}
\newcommand{\bary}{\bar{y}}

\newcommand{\barF}{\bar{F}}

\newcommand{\barX}{\bar{X}}

\newcommand{\ip}[2]{\left\langle{#1},{#2}\right\rangle}
\newcommand{\ipt}[2]{\langle{#1},{#2}\rangle}

\newcommand{\lranglet}[2]{\langle{#1},{#2}\rangle}

\newcommand{\eqa}{\stackrel{\rm(a)}{=}}

\DeclareMathOperator*{\argmax}{arg\,max}
\DeclareMathOperator*{\argmin}{arg\,min}

\DeclareMathOperator{\st}{s.t.}

\DeclareMathOperator{\tr}{tr}

\DeclareMathOperator{\rank}{rank}

\newtheorem{theorem}{Theorem} 
\newtheorem*{theorem*}{Theorem}
\newtheorem{lemma}{Lemma}

\newtheorem{prop}{Proposition}
\newtheorem{corollary}{Corollary}
\newtheorem{assump}{Assumption} 
\newtheorem*{assump*}{Assumption}

\theoremstyle{definition}
\newtheorem{definition}{Definition} 

\theoremstyle{remark}
\newtheorem{remark}{Remark}

\newcommand{\qednew}{\nobreak \ifvmode \relax \else
      \ifdim\lastskip<1.5em \hskip-\lastskip
      \hskip1.5em plus0em minus0.5em \fi \nobreak
      \vrule height0.75em width0.5em depth0.25em\fi}

\usepackage[letterpaper,left=1in,right=1in,top=1in,bottom=1in]{geometry}

\title{The Generalized Multiplicative Gradient Method for A Class of Convex Optimization Problems Over Symmetric Cones}
\usepackage[colorlinks=true,breaklinks=true,bookmarks=true,urlcolor=blue,
     citecolor=blue,linkcolor=blue,bookmarksopen=false,draft=false]{hyperref}
\author{Renbo Zhao\thanks{Department of Business Analytics, Tippie College of Business, University of Iowa (\href{mailto:renbo-zhao@uiowa.edu}{renbo-zhao@uiowa.edu}).}
}

\newcommand{\blue}[1]{\textcolor{blue}{#1}}
\usepackage{tabularx}

\numberwithin{equation}{section}
\numberwithin{lemma}{section}
\numberwithin{prop}{section}
\numberwithin{definition}{section}
\numberwithin{assump}{section}
\numberwithin{remark}{section}
\numberwithin{corollary}{section}
\numberwithin{theorem}{section}

\usepackage{nccmath}

\newcommand{\PI}{\mathrm{PI}}
\newcommand{\GL}{\mathrm{GL}}
\newcommand{\Aut}{\mathsf{Aut}}

\newcommand{\R}{\mathsf{R}}

\begin{document}

\maketitle

%
%

\abstract{
We develop and analyze the Generalized Multiplicative Gradient (GMG) method for solving a class of convex optimization problems over symmetric cones, where the objective function does not have Lipschitz gradient over the feasible region. This problem class includes several applications, such as positron emission tomography, D-optimal design, quantum state tomography and the dual problem of Nesterov's convex relaxation of the boolean quadratic problem. We show that the GMG method has a convergence rate of $O(1/k)$ in terms of the objective gap. Our analysis of the convergence rate is rather unconventional, and to that end, 
we establish several results that may be of independent interest, such as a curvature bound of the Legendre and logarithmically-homogeneous functions,  and a Cauchy–Schwarz inequality in  representative simple Euclidean Jordan Algebras. Finally, we compare the computational complexity of the GMG method  with three other related first-order methods on several important applications, and we show that under certain mild assumptions, the GMG method achieves the best (or almost the best) computational complexities on all of the applications. 
 
}

\section{Introduction}\label{sec:intro}

Let us consider the following convex optimization problem: 
\begin{align}
F^*:= \max_{x\in\Delta_n}\; \left\{F(x):=\textstyle\sum_{j=1}^m p_j\ln(a_j^\top x)\right\}, \tag{PET} 
\label{eq:PET}
\end{align}
where  $p_j> 0$ and $a_j\ge 0$ 
for all $j\in [m]$, $\sum_{j=1}^m p_j=1$,  
and 
$$\Delta_n :=\{x \in \bbR^n : x \ge 0,\;\;  \textstyle \sum_{i=1}^n x_i = 1\}$$
denotes the unit simplex in $\bbR^n$. Furthermore, we assume that  $a_j\ne 0$ for all $j\in[m]$ and  for any $i\in[n]$, 
$\sum_{j=1}^m(a_j)_i>0$. 
 This problem has a long history and arises in different contexts, such as 
the positron emission tomography (PET) in medical imaging~\cite{Vardi_85}, the log-optimal investment~\cite{Cover_84},  and  the  mixture model estimation in statistics~\cite{Vardi_93}. Although this problem seems to have a simple structure, it poses nontrivial challenges to the ``classical'' first-order methods (see e.g.,~\cite{Nest_04}), since the objective function $F$, despite being differentiable, has non-Lipschitz gradient over the constraint set $\Delta_n$. 
Interesting enough, back in 1980s, Cover~\cite{Cover_84} proposed the following extremely simple but ``non-standard'' 
multiplicative gradient (MG) method to solve~\eqref{eq:PET}: starting from $x^0\in \ri \Delta_n$ (i.e., the relative interior of $\Delta_n$), at each iteration $t\ge 0$, one obtains the next iterate $x^{t+1}$ by (element-wise) multiplying the current iterate $x^t$ with 
$\nabla F(x^t)$, namely 
\begin{equation}
x^{t+1}:= x^t \odot \nabla F(x^t) \quad \Longleftrightarrow\quad  x_i^{t+1}:= x_i^t \nabla_i F(x^t), \quad \forall\,i\in[n]. \tag{MG}\label{eq:MG}
\end{equation}
Cover showed that $F(x^t)\uparrow F^*$ as $t\to +\infty$ and Csisz\'ar and Tusn\'ady~\cite{Csiszar_84} further showed that the sequence $\{x^t\}_{t\ge 0}$ has a unique limit point (which is an optimal solution of~\eqref{eq:PET}). 
Despite the thorough understanding of the asymptotic convergence properties of~\eqref{eq:MG} for solving~\eqref{eq:PET}, the convergence rate of this algorithm remains unknown till recently. 
Indeed, our recent work~\cite{Zhao_23pet} showed that when $x^0 = (1/n) e$ (where $e$ denotes the vector of all ones), we have the following $O(1/t)$ convergence rate for both the last- and averaged-iterate: 
\begin{equation}
\max\big\{F^* - F(x^t), F^* - F(\barx^t)\big\} \le \frac{\ln(n)}{t+1}, \quad{\rm where} \;\; \barx^t:=\tfrac{1}{t+1}\textstyle 
{\sum}_{k=0}^t \;x^k,\quad \forall\, t\ge 0. \label{eq:rate_PET}
\end{equation}


At this point, one might suspect that~\eqref{eq:MG} is a tailor-made method for solving~\eqref{eq:PET}, and bears no generality. 
Let us present another 
important problem in computational statistics, namely the D-optimal design problem, where~\eqref{eq:MG} demonstrates its efficacy.  
Let $\bbS_{+}^m$ and $\bbS_{++}^m$ denote the sets of $m\times m$  symmetric positive semi-definite (PSD) and  positive definite matrices, respectively, and we write $A\succeq 0$ if $A \in \bbS_{+}^m$ and $A\succ 0$ if $A \in \bbS_{++}^m$. Given matrices $\{A_i\}_{i=1}^n \subseteq \bbS_{+}^m$ such that $A_i\ne 0$ for $i\in[n]$ and $\sum_{i=1}^n A_i \succ 0$, the D-optimal design problem is written as 
\begin{equation}
 {\max}_{x\in\Delta_n }\; (1/m) \textstyle \ln\det\big(\sum_{i=1}^n x_i A_i\big). 
\tag{DOPT} \label{eq:DOPT} 
\end{equation}
From the viewpoint of statistical estimation,  
solving~\eqref{eq:DOPT} corresponds to finding a design measure that maximizes the log-determinant of the expected Fisher information matrix about the unknown model parameter 
 that we wish to estimate (for details, see e.g.,~\cite{Pukel_06,Yu_10,Zhao_25}). 
Note that depending on the underlying statistical model, the matrices $\{A_i\}_{i=1}^n$ need not have rank one. However, when they do,~\eqref{eq:DOPT} also arises as the Lagrangian dual of the minimum-volume covering ellipsoid  problem in computational geometry, which dates back at least 70 years to \cite{john} (see also~\cite{toddminimum}). Somewhat surprisingly, in the optimal design literature,~\eqref{eq:MG} is one of the most widely adopted methods for solving~\eqref{eq:DOPT}. This method was first introduced in~\cite{Silvey_78}, and received extensive research efforts over the past fifty years (see~\cite{Yu_10} and references therein). Similar to the situation as in~\eqref{eq:PET}, the asymptotic convergence of~\eqref{eq:MG} when applied to~\eqref{eq:DOPT} has been well-understood. Specifically, from~\cite[Theorem~2]{Yu_10}, we know that $F(x^t)\uparrow F^*$ as $t\to +\infty$ and all limit points of $\{x^t\}_{t\ge 0}$ are optimal for~\eqref{eq:DOPT}. The convergence rate of~\eqref{eq:MG} remains unknown until the recent paper~\cite{Cohen_19}. Specifically,  the authors showed that when each $A_i$ has rank one (for $i\in[n]$) and $x^0 = (1/n) e$, we have 
\begin{equation}
{\max}_{i=1}^n\; \ln(\nabla_i F(\barx^t))\le \frac{\ln(n)}{t+1}, \quad \forall\, t\ge 0,
\end{equation}
where $\barx^t$ is defined in~\eqref{eq:rate_PET}.  Despite this important progress, it is still unclear whether we have the $O(1/t)$ convergence rate in terms of the objective gap for the general case where $\{A_i\}_{i=1}^n$ do not have rank one. In this work, we will provide an affirmative answer to this question.

One may observe that the optimization variables in both~\eqref{eq:PET} and~\eqref{eq:DOPT} 
are vectors living in the unit simplex $\Delta_n$. Quite interestingly, a recent work~\cite{Lin_21} modified~\eqref{eq:MG} to solve  the quantum state tomography (QST) 
problem (see e.g.,~\cite{Hra_04}). Here the optimization variable is a 
matrix living in the spectraplex $\calS_n:=\{X\in\bbH_+^n:\tr(X)= 1\}$, where $\bbH_+^n$ denotes the cone of  complex Hermitian PSD  matrices. Given matrices  $\{A_j\}_{j=1}^m\subseteq \bbH_+^n\setminus\{0\}$ such that $\sum_{j=1}^m A_j \succ 0$, and some probability vector $p
\in\ri\Delta_m$, the QST problem reads 
\begin{align}
 F^*={\max}_{X\in\calS_n}\;  \left\{F(X):=\textstyle \sum_{j=1}^m p_j\ln\big(\ipt{X}{A_j}\big)\right\}.  \tag{QST} \label{eq:QST}
\end{align} 
The modified MG method in~\cite{Lin_21} starts with $X^0 = (1/n)I_n$, and iterates as follows: 
\begin{equation}
\hatX^{t+1}:= \exp\big(\ln(X^t)+\ln(\nabla F(X^t))\big),\quad 
X^{t+1}:= { \hatX^{t+1}}/{\tr(\hatX^{t+1})}, \quad \forall\,t\ge 0,  \label{eq:MG_matrix}
\end{equation}
where $\exp(\cdot)$ and $\ln(\cdot)$ denote matrix exponential and logarithm, respectively. In addition, the authors showed that this method has the following  $O(1/t)$ convergence rate: 
\begin{equation}
F^* - F(\barX^t) \le \frac{\ln(n)}{t+1},  \quad{\rm where} \;\; \barX^t:=\tfrac{1}{t+1}\textstyle 
\sum_{k=0}^t X^k,\quad\forall\, t\ge 0. \label{eq:rate_PET}
\end{equation}
 
 

In summary, the recent progress on the convergence-rate analysis of~\eqref{eq:MG} and its variant (henceforth referred to as the ``MG-type methods'') provides us 
with important computational guarantees of these methods on
~\eqref{eq:PET}, \eqref{eq:DOPT} and~\eqref{eq:QST}. Note that the objective function in all of the three problems have non-Lipschitz gradient over the constraint set, and as mentioned previously, this makes them challenging to be solved by the classical first-order methods (FOMs).  On the other hand,~\eqref{eq:MG} and its variant seem quite effective in solving these ``challenging'' problems. This intriguing observation motivates us to systematically investigate into   the MG-type methods, and develop a general theory for them. Specifically, we aim to 
 formulate a class of convex problems that include~\eqref{eq:PET}, \eqref{eq:DOPT} and~\eqref{eq:QST} as special instances, and  propose and analyze a generalized MG (GMG) method on this class of problems. Ideally, on the aforementioned three specific problems, 
 the GMG method and its computational guarantees 
 should 
 recover  the specific MG-type methods and their computational guarantees above.  


Next, we shall formulate our problem class and introduce another application that falls under this problem class. Before doing so, let us make an important remark to distinguish the MG method above applied to~\eqref{eq:PET} and~\eqref{eq:DOPT}  
 from  the entropic mirror descent (EMD) method 
 (see e.g.,~\cite[Lemma~4]{Nest_05}).  
Although both methods have ``multiplicative'' forms, their details are vastly different. 
Specifically, in each iteration of EMD, the iterate is (element-wise) multiplied with the {\em exponentiated} gradient, and a {\em normalization} step is needed to bring the product back to the unit simplex $\Delta_n$. Furthermore, the computational guarantees of EMD require the objective $F$ to either  have Lipschitz gradient (or function value) over $\Delta_n$ or be ``relatively smooth'' with respect to (w.r.t.) the {\em (negative) entropy function}~\cite{Bauschke_17,Lu_18}, but it turns out that neither condition holds for the objectives in~\eqref{eq:PET} and~\eqref{eq:DOPT}. In fact,  careful thinking reveals that the MG-type methods above not only are {\em fundamentally} different from EMD, but also do not trivially fall under  any well-established optimization paradigms. 





\subsection{Problem Class}\label{sec:prob}

Let $\bbY$ be a finite-dimensional real vector space (with dual space $\bbY^*$), and define $\barbbR:= \bbR\cup\{+\infty\}$.

\begin{definition}[{Legendre  function;~\cite[Definition~2.1]{Bauschke_00}}] \label{def:legendre}
Let $h:\bbY\to \barbbR$ be a proper, 
closed and convex function with $\inter\dom h\ne \emptyset$, where $\dom h:= \{y\in\bbY: h(y)<+\infty\}$. We call $h$ {Legendre} if it is 
 i) {\em essentially smooth}, namely   it is (continuously) differentiable on $\inter\dom h$, and 
 if $\bdry\dom h\ne \emptyset$, then  for any $\{x^k\}_{k\ge 0}\subseteq \inter\dom h$ such that $x^k\to x\in \bdry\dom h$, we have $\normt{\nabla h(x^k)}_*\to +\infty$,
  and ii)  {\em essentially strictly convex}, namely it is strictly convex on $\inter\dom h$.  
\end{definition}

\begin{definition}[{Logarithmically homogeneous function}]\label{def:LH}
Let $h:\bbY\to \barbbR$ be a proper, closed and convex function. We call it $\theta$-logarithmically homogeneous ($\theta$-LH) for some $\theta>0$, if 
\begin{equation}
h(ty) = h(y) - \theta\ln t, \quad \forall\, y\in \dom h,\;\; \forall\, t>0.  \label{eq:LH}
\end{equation}
\end{definition}

Let $(\bbV,\ipt{\cdot}{\cdot})$ be a finite-dimensional real inner-product space. 
Our problem class reads 
\begin{align}
\begin{split}
F^*:= {\max}\;\; &\big\{F(x):= f(\rvA x)\big\}\quad \\
\st\;\;  &x\in\calC:=\{x\in\calK:\tr(x) 
 = 1\}.
\end{split} \tag{P} \label{eq:P}
\end{align}
In~\eqref{eq:P}, $-f:\bbY\to \barbbR$ is a Legendre and $\theta$-LH function for some $\theta>0$ (which is proper, closed and convex by Definition~\ref{def:legendre}), $\rvA:\bbV\to \bbY$ is a linear operator, and $\calK\subseteq \bbV$ is a {\em symmetric cone}, namely it is self-dual and homogeneous. The Jordan algebraic characterization of symmetric cones indicates that there exists a unique Euclidean Jordan algebra  associated with $\calK$, and the trace 
$\tr(\cdot)$ in the description of $\calC$ is defined w.r.t.\ this algebra. 
(We will provide more details in Section~\ref{sec:background}.) 
Typically examples of $\calK$ arising in optimization problems include the nonnegative orthant $\bbR_+^n$, the second-order cone, the real symmetric PSD cone $\bbS_+^n$, the complex Hermitian PSD cone $\bbH_+^n$ and their Cartesian products. In addition, for $x\in \bbR_+^n$, 
$\tr(x):= \sum_{i=1}^n x_i$, and for $X\in \bbS_+^n$ or $\bbH_+^n$, $\tr(X):= \sum_{i=1}^n \lambda_i(X)$, where $\{\lambda_i(X)\}_{i=1}^n$ denote the (real) eigenvalues of $X$.  

We will present and analyze our GMG method for solving~\eqref{eq:P} in Section~\ref{sec:algo}. 
To ensure that~\eqref{eq:P} is well-posed and  this method is well-defined, and also for the analysis of this method, 
we need two additional assumptions on $\rvA$ and $F$. These two assumptions will 
appear in Assumptions~\ref{assump:A} and~\ref{assump:F}, after we introduce some analytic properties of $f$ in Section~\ref{sec:legendre_LH}. For now, 
let us illustrate how to put the three specific problems~\eqref{eq:PET}, \eqref{eq:DOPT} and~\eqref{eq:QST} into the general problem template~\eqref{eq:P}. For~\eqref{eq:PET}, $f:y\mapsto \sum_{j=1}^m p_j\ln(y_j)$ for $y>0$ (and $\theta=\sum_{j=1}^m p_j=1$), $\rvA :x\mapsto (a_j^\top x)_{j=1}^m$ and  $\calK = \bbR_+^n$. For~\eqref{eq:DOPT},  $f:Y\mapsto (1/m) \textstyle \ln\det (Y)$ for $Y\succ0$ (and $\theta=1$), $\rvA :x\mapsto \sum_{i=1}^n x_i A_i$, and $\calK = \bbR_+^n$. For~\eqref{eq:QST}, $f$ is the same as in~\eqref{eq:PET}, $\rvA :X\mapsto (\ipt{A_j}{X})_{j=1}^m$ and  $\calK = \bbH_+^n$. In addition, it is easy to see that the function $-f$ in all of the three problems is indeed Legendre and LH.

\subsection{Another Application of~\eqref{eq:P} } \label{sec:another_app}

This application arises as the dual problem of Nesterov's convex relaxation  of the the Boolean quadratic program (BQP), namely 
\begin{equation}
q^*:= {\max}_{x\in\{\pm 1\}^n}\; x^\top  A x, \quad\where \;\; A\in \bbS_{++}^n, \tag{BQP}
\end{equation}
 and is well-known to be 
NP-hard.  Nesterov~\cite{Nest_98a} showed that the semidefinite relaxation 
\begin{equation}
s^*:={\min}_{y} \;\; \textstyle\sum_{i=1}^n y_i \quad\st \;\;\diag(y)\succeq A \label{eq:SDP_relax}
\end{equation}
provides a $(2/\pi)$-approximation of the BQP, namely $(2/\pi)s^*\le q^*\le s^*$. 
Let $A = LL^\top$ be the Cholesky decomposition of $A$, and let $q_i$ denote the $i$-th column of $L^\top$ for $i\in[n]$. 
In~\cite[Lemma~6]{Nest_11}, 
Nesterov showed that the dual problem of~\eqref{eq:SDP_relax} 
can be equivalently written as 
\begin{equation}
{\max}_{X} \quad   2\ln \textstyle\big(\sum_{i=1}^n \ipt{X}{q_iq_i^\top}^{1/2}\big) \quad  \st\;\; X\in\bbS_+^n, \; \tr(X)=1. \tag{DBQP} \label{eq:DBQP}
\end{equation}
We can fit~\eqref{eq:DBQP} into the problem class~\eqref{eq:P}, by taking $f:y\mapsto 2\ln(\sum_{i=1}^n \sqrt{y_j})$ for $y\in\bbR_{+}^n\setminus\{0\}$ (and $\theta=1$), $\rvA:X\mapsto (\ipt{X}{q_iq_i^\top})_{i=1}^n$ and $\calK = \bbS_+^n$. Also, 
we can easily check that $\barf := -f$ in this case is Legendre and $1$-LH. 
That said, 
note that $\barf$ in~\eqref{eq:DBQP} 
behaves rather differently on 
$\bdry \dom \barf$ 
than that in~\eqref{eq:PET}, \eqref{eq:DOPT} and~\eqref{eq:QST}. 
 Specifically, in those problems, $\barf$ is a convex 
 barrier of $\cl\dom \barf$, namely, $\dom \barf$ is open and $\barf(y^k)\to+\infty$ as $y^k\to y\in \bdry\dom \barf$.
 In contrast, $\barf$ in~\eqref{eq:DBQP} does not possess such a ``barrier'' property, and 
 $\dom \barf$ is neither open nor closed. This makes~\eqref{eq:DBQP} 
 harder to solve by FOMs than the other three problems. In fact, to our knowledge,  the only  FOM that can solve~\eqref{eq:DBQP} with computational guarantees is Nesterov's Barrier Subgradient Method (BSM)~\cite{Nest_11}, whose convergence rate is $O(\ln(t)/\sqrt{t})$.  As we shall see later, our GMG method can solve~\eqref{eq:DBQP} with convergence rate $O(1/{t})$, which is much faster compared to BSM.


\subsection{Main Contributions}

In this work, we develop and analyze the GMG method (cf.~Algorithm~\ref{algo:GMG}) for solving the problem~\eqref{eq:P}. When specialized to~\eqref{eq:PET}, \eqref{eq:DOPT} and~\eqref{eq:QST}, this method recovers the MG-type methods introduced previously. 
We show that with a simple choice of the starting point, this method achieves a convergence rate of $O(\theta\ln(n)/t)$ in terms of the objective gap, where $n$ denotes the rank of the Euclidean Jordan algebra  associated with $\calK$. As a by-product of our analysis, we obtain an interesting upper bound on the objective gap of any point $x\in\ri\calC$ in~\eqref{eq:P} (cf.~Remark~\ref{rmk:x^0_rate}). 

In fact, to our knowledge, our analysis of the GMG method differs vastly from the analysis of almost any FOM in the literature. To conduct our analysis, we establish several results that may be of independent interest. Specifically, in  Section~\ref{sec:legendre_LH}, we derive several convex analytic properties of a Legendre and LH function $h$, 
including a curvature bound of $h$ (cf.~\eqref{eq:curv_ub}) that plays an important role in analyzing  the GMG method. In addition, our  analysis also requires the log-gradient of $F$ 
to be $\calK$-convex on  $\inter\calK$ (cf.\ Assumption~\ref{assump:F}), an assumption that is rather unconventional in the literature of FOMs (or even higher-order methods). In fact, a major focus of this work is to identify several classes of functions that satisfy  this assumption (cf.\ Section~\ref{sec:function}), and hence demonstrate the wide applicability of this assumption. 
To that end, we establish a Cauchy–Schwarz inequality in  representative simple Euclidean Jordan Algebras~\cite{Schm_01}, which is a non-trivial generalization of the ``matrix Cauchy-Schwarz inequality'' (see e.g.,~\cite{Lav_08,Zhao_25}).  

Lastly, in Section~\ref{sec:comp_complexity}, we compare  the computational complexity of the {GMG} method  with three other related FOMs on the four problems introduced above, namely~\eqref{eq:PET}, \eqref{eq:DOPT},~\eqref{eq:QST} and~\eqref{eq:DBQP}.  All of the three FOMs under comparison were developed to solve  convex optimization problems where the objective function lacks the Lipschitz-gradient condition (over the feasible region), but each focuses a different problem class than~\eqref{eq:P}. We show that under certain mild assumptions, the GMG method achieves the best (or almost the best) computational complexities on all of the four problems. 

\subsection{Notations} \label{sec:notations}
Define $\bbR_{+}:=[0,+\infty)$ and $\bbR_{++}:=(0,+\infty)$. Let $\bbS^n$ and $\bbH^n$ denote the set of $n\times n$ real symmetric and complex Hermitian matrices, respectively, and $I_n$ denote the identity matrix of size $n\times n$. Given $\emptyset\ne \calA\subseteq \bbY$, let $\ri\calA$, $\inter\calA$, $\cl\calA $ and $\bdry\calA$ denote the relative interior, interior, closure and boundary of $\calA$, respectively. We call $\calA$ solid if $\inter\calA\ne \emptyset$.  A nonempty closed convex cone $\Omega\subseteq\bbY$ is pointed if $\Omega\cap(-\Omega)=\{0\}$. Define the dual cone of $\Omega$ by $\Omega^*:= \{z\in\bbY^*:\ipt{z}{y}\ge 0,\,\forall\,y\in\Omega\}$, where $\ipt{\cdot}{\cdot}:\bbY^*\times \bbY\to \bbR$ denotes the duality pairing between $\bbY^*$ and $\bbY$. If $\bbY$ is equipped with an inner product, then $\ipt{\cdot}{\cdot}$ is identified with this inner product. In addition, we define the polar cone of $\Omega$ by $\Omega^\circ := -\Omega^*$.  We call $\Omega$ {\em regular} if it is closed, convex, solid and pointed. In this case, $\Omega$ induces a partial order $\succeq_\Omega$ on $\bbY$, 
namely, $a\succeq_\Omega b$ if and only if $a-b \in \Omega$. In addition, $\Omega$ also induces a strict partial order $\succ_\Omega$ on $\bbY$, namely, $a\succ_\Omega b$ if and only if $a-b \in \inter\Omega$. 





\section{Background on Euclidean Jordan algebra} 
\label{sec:background}


For the development and analysis of our GMG method, we briefly provide some background on Euclidean Jordan algebra and symmetric cones. \\[-2ex]

\noindent
{\bf Power-associative algebra.} Let $\bbV$ be a finite-dimensional real vector space endowed with a {\em bilinear} product $\circ:\bbV\times\bbV\to\bbV$, and we call $(\bbV,\circ)$ an  $\bbR$-algebra. 
If there exists $e\in\bbV$ such that $x\circ e = e\circ x = x$ for all $x\in\bbV$, then $e$ is the identity element of $(\bbV,\circ)$.   
If $(\bbV,\circ)$ has identity element and 
$(x\circ x)\circ x = x\circ(x\circ x)$ for all  $x\in \bbV$, then  $(\bbV,\circ)$  is called power associative, 
in which case  we can define   $x^k:= x\circ \cdots \circ x$ ($k$  copies of  $x$) for $k\ge 1$ without ambiguity. We also define $x^0:=e$. In addition, for $x\in\bbV$, let $\bbR[x]$ denote the set of real polynomial in $x$, and note that $(\bbR[x],\circ)$ is a commutative and associative $\bbR$-algebra. We call $x\in\bbV$ {\em invertible}, if there exists 
$y\in \bbR[x]$ such that $y\circ x = e$. 
By the associativity of $(\bbR[x],\circ)$, $y$ is unique. In this case, we call $y$ the inverse of $x$ and denote it by $x^{-1}$.

Let  $(\bbV,\circ)$  be power associative. For each  $x\in \bbV$,  let  $k$  be the smallest positive integer such that the set  $\{e,x,x^2,\ldots ,x^k\}$  is linearly dependent, and we call $k$ the degree of  $x$, 
denoted by $\deg (x) $. Define  $\rank(\bbV):= \max\{\deg(x):x\in\bbV\}$. 
 An element  $x\in\bbV$  is regular if $\deg(x) = \rank(\bbV)$, and we denote the set of regular elements in $\bbV$ by $\calR(\bbV)$. 
  
 Let $\rank(\bbV)=r$. For any $x\in \bbV$, let $m_x$ denote the minimal polynomial of $x$, namely the (unique) {\em monic}  polynomial of the lowest degree such that $m_x(x)=0$. Clearly, the $\deg(m_x) = \deg(x)$, where $\deg(m_x)$  denotes the degree of $m_x$. From~\cite[Proposition~II.2.1]{Faraut_94}, we know that  there exist unique polynomials $a_1,\ldots,a_r:\bbV\to \bbR$ 
 such that $a_k$ is homogeneous of degree $k$ ($k\in[r]$), and 
 for any $x\in \calR(\bbV)$, its minimal polynomial in variable $\lambda$ is   
 \begin{equation}
 m_x(\lambda) = c_x(\lambda) := \lambda^r - a_1(x)\lambda^{r-1} + a_2(x)\lambda^{r-2}+\cdots + (-1)^ra_r(x), 
 \end{equation}
 where $c_x$ is called the {\em characteristic polynomial} of $x\in\calR(\bbV)$. 
 Furthermore, we know that $\calR(\bbV)$ is open and dense, and by the continuity of $\{a_i\}_{i=1}^r$ on $\bbV$, we can extend the definition of  $c_x$ to all $x\in \bbV$. 
 (In fact, for non-regular $x\in\bbV$, $m_x$ divides $c_x$.) The roots of $c_x$ are called the {\em eigenvalues} of $x$ and denoted by $\{\lambda_i(x)\}_{i=1}^r$. 
 In addition, we define the trace and determinant of $x$ as follows: 
 \begin{equation}
\tr(x):= a_1(x)=\textstyle\sum_{i=1}^r \lambda_i(x)   \quad\mbox{and}\quad \det(x):= a_r(x)=\textstyle\prod_{i=1}^r \lambda_i(x)  .  \label{eq:def_det}
\end{equation}
Note that since $a_1$ is a homogeneous polynomial of degree 1, $\tr(\cdot)$ is a linear function on $\bbV.$ In addition, from~\cite[pp.~29]{Faraut_94}, we know that $c_x(\lambda)= \det(\lambda e - x)$ for all $\lambda$ and all $x\in\bbV$. 
\\[-2ex]



\noindent 
{\bf Jordan algebra}.\ Let $(\bbV,\circ)$ be a finite-dimensional $\bbR$-algebra with identity $e$.
{\em Throughout this work, all the algebras that we will deal with 
are assumed to be of this type.} 
 We call it  
a Jordan algebra if for all $x,y\in\bbV$, we have (i) $x\circ y = y\circ x$ and (ii) $x^2 \circ(x\circ y) = x \circ(x^2\circ y)$. 
Clearly, a Jordan algebra is {\em commutative} and {\em power associative}. 
Since $\circ$ is bilinear, for $x\in\bbV$, 
let $L(x):\bbV\to\bbV$ be  the linear operator 
such that $L(x)y:= x\circ y$ for all $y\in\bbV$, and note that $L(\cdot)$ is linear on $\bbV$. 
We call $L(x)$ the {\em linear representation} of $x\in\bbV$. Based on $L(\cdot)$, for any $x\in\bbV$, 
we define the linear operator 
\begin{equation}
P(x):= 2L(x)^2 - L(x^2):\bbV\to\bbV, \label{eq:def_P} 
\end{equation}
as the {\em quadratic representation} of $x$. 
\\[-2ex]

%
%
%
%
%

\noindent 
{\bf Euclidean Jordan algebra}.\ A 
Jordan algebra $(\bbV,\circ)$ 
is called {\em Euclidean} if $\tr(x^2)>0$ for any $x\in\bbV\setminus\{0\}$. 
In this case, the symmetric bilinear form $(x,y)\mapsto \tr(x\circ y)$ 
 is {\em positive definite}, and hence   becomes  an inner product on $\bbV$, which we denote by $\ipt{\cdot}{\cdot}$, namely
\begin{equation}
\ipt{x}{y}:= \tr(x\circ y), \quad \forall\,x,y\in\bbV. \label{eq:tr_ipt}
\end{equation}
(In particular, $\tr(x) = \ipt{e}{x}$ for $x\in\bbV$.) 
In addition, from~\cite[Proposition~II.4.3]{Faraut_94}, we know that the inner product $\ipt{\cdot}{\cdot}:\bbV\times\bbV\to\bbR$ defined in~\eqref{eq:tr_ipt} is {\em associative}, namely   
\begin{equation}
\ipt{x}{y\circ z} = \ipt{x\circ y}{z}, \quad \forall\,x,y,z\in\bbV. 
\end{equation}
This implies that for any  $x\in\bbV$,  both the linear operators $L(x)$ and $P(x)$ are {\em self-adjoint} w.r.t.\ the inner product $\ipt{\cdot}{\cdot}$.
Let $\normt{\cdot}$ be the norm induced by $\ipt{\cdot}{\cdot}$, namely $\normt{x}:= {\ipt{x}{x}}^{1/2}$ for  $x\in\bbV$. 

Let $(\bbV,\circ)$ 
be a Euclidean Jordan algebra with rank $r$. We call $q\in\bbV$ an  idempotent if $q^2 = q$, and two idempotents $q$ and $q'$ orthogonal if $q\circ q'=0$. We call 
$\{q_i\}_{i=1}^k\subseteq\bbV$ (where $k\in [r]$) a complete system of  orthogonal idempotents (CSOI) if (i)  $q_i^2 = q_i$ and  $q_i\ne 0$ for all $i\in[k]$, (ii) $q_i\circ q_j = 0$ for all $i\ne j$,  $i,j\in[k]$ and (iii) $\sum_{i=1}^k q_i=e$. From~\cite[Proposition~III.1.1]{Faraut_94}, we know that for any $x\in\bbV$ with degree $k\in [r]$, there exist {\em  unique distinct real} numbers $\{\mu_j\}_{j=1}^k$ and a {\em unique} CSOI $\{c_j\}_{j=1}^k$ such that $x = \sum_{j=1}^k \mu_j c_j$. 
This is called the  
{\em  type-I spectral decomposition} of $x$, 
and $\{\mu_j\}_{j=1}^k$  are the distinct (real) eigenvalues of $x$. 

In addition, we call an idempotent $q\in\bbV$ {\em primitive} if it is non-zero and cannot be written as the sum of two non-zero idempotents. We call $\{q_i\}_{i=1}^r$ a {\em Jordan frame} if it is a CSOI and each $q_i$ is primitive, for $i\in[r]$. From~\cite[Proposition~III.1.1]{Faraut_94}, we know that for any $x\in\bbV$, there exist real numbers $\{\lambda_i\}_{i=1}^r$ and a Jordan frame $\{q_i\}_{i=1}^r$ such that $x = \sum_{i=1}^r \lambda_i q_i$. 
This is a   {\em  type-II spectral decomposition} of $x$ (and may not be unique). 
Note that $\{\lambda_i\}_{i=1}^r$ are the (real) eigenvalues of $x$ including multiplicities. In addition, we call $\{q_i\}_{i=1}^r$ the Jordan frame of $x$. 
In fact, one can observe a simple connection between the type-I and type-II spectral decompositions: let $x\in\bbV$ 
with  $\deg(x) = k$, then 
there exists a partition $\{\calI_j\}_{j=1}^k$ of $[r]$ such that for any $j\in[k]$, $\lambda_i = \mu_j$ for all $i\in\calI_j$ and $c_j = \sum_{i\in\calI_j} q_i$. The uniqueness of the type-I spectral decomposition then implies that $\{\lambda_i\}_{i=1}^r$ 
are uniquely determined by $x$.

 In the sequel, 
 any spectral decomposition we mention
will be of type II, unless otherwise stated.
From above, we easily see that 
i) the all the $r$ eigenvalues of $e$ 
are equal to one, and hence $\tr(e)= r$ and $\det(e) = 1$, and (ii) any primitive idempotent $q$ has one unit eigenvalue  and $r-1$ zero eigenvalues, and hence $\normt{q}^2 = \tr(q^2)= \tr(q)= 1$ and $\det(q) = 0$. 
\\[-2ex] 
 
 \noindent 
{\bf Functions of $x\in\bbV$}.\ 
 The spectral decomposition above enables us to extend the definition of any univariate 
 function $h:\calD_h\to \bbR$ (where $\calD_h\subseteq \bbR$) to any element $x\in\bbV$: 
 let $x= \sum_{i=1}^r \lambda_i q_i$ be a spectral decomposition of $x$, if  $\{\lambda_i\}_{i=1}^r\subseteq \calD_h$, then 
\begin{equation}
h(x) := \textstyle\sum_{i=1}^r h(\lambda_i) q_i.   \label{eq:func_x}
\end{equation}
By the connection between the type-I and type-II spectral decompositions, we know that $h(x)= \textstyle\sum_{j=1}^k h(\mu_j) c_j$, and hence is uniquely defined. \\[-1.5ex]

\noindent 
{\bf Symmetric cones}.\ Let $(\bbV,\ipt{\cdot}{\cdot})$ be a (finite-dimensional) real inner-product space.  
We call $\calK$ {\em self-dual} if $\calK = \calK^*$ (where $\calK^*$ is defined w.r.t.\ the inner product $\ipt{\cdot}{\cdot}$). Note that a self-dual cone is {\em regular} (cf.~Section~\ref{sec:notations}). 
In addition, let $\Aut(\calK)$ denote the set of linear automorphisms on $\calK$, namely 
$\Aut(\calK):=\{{\rvL}\in\GL(\bbV):\rvL(\calK)=\calK\},$ 
 where $\GL(\bbV)$ denotes the general linear group of $\bbV$. We call a solid cone $\calK$ homogeneous if for all $x,y\in\inter\calK$, there exits $\rvL\in \Aut(\calK)$ such that $\rvL x  = y$. We call $\calK$ symmetric if it is both self-dual and homogeneous.  Note that by the self-duality of $\calK$, we have $\ipt{x}{y}\ge 0$ for all $x,y\in\calK$.

Let $(\bbV,\circ)$ 
be a Euclidean Jordan algebra. Define its cone of squares $\calK(\bbV):=\{x^2:x\in\bbV\}$. From~\cite[Theorem~III.2.1]{Faraut_94}, we know that $\calK(\bbV)$ is symmetric. Conversely, from~\cite[Theorem~III.3.1]{Faraut_94}, we know that any symmetric cone is the cone of squares of some Euclidean Jordan algebra. Therefore, we have the {\em Jordan algebraic characterization of symmetric cones}: a cone is symmetric if and only if it is the cone of squares of some Euclidean Jordan algebra. In fact, there exists a one-to-one correspondence between symmetric cones and Euclidean Jordan algebras. In addition, 
we know that  $\calK(\bbV)= \{x\in \bbV:\lambda_{\min}(x)\ge 0\}$ and $\inter\calK(\bbV)= \{x\in \bbV:\lambda_{\min}(x)> 0\}$, where $\lambda_{\min}(x)$ denotes the minimum eigenvalue of $x\in \bbV$ (see e.g.,~\cite[Proposition~2.5.10]{Vieira_07}). As a result, we know that $e\in\inter\calK$ and $q\in\calK$, where $q$ is any primitive idempotent in $\bbV$. 
 \\[-2ex]

\noindent 
{\bf Simple Euclidean Jordan algebras and irreducible symmetric cones}.\ Let $(\bbV,\circ)$ be a Jordan algebra. We call a linear subspace $\bbD\subseteq\bbV$ an ideal of $\bbV$ if for all $x\in \bbD$ and $y\in \bbV$, $x\circ y\in \bbD$. We call $(\bbV,\circ)$ {\em simple} if $\bbV\ne \{0\}$ and the only ideals are $(\{0\},\circ)$ and $(\bbV,\circ)$. Let $\{(\bbV_i,\circ)\}_{i=1}^n$ be Jordan sub-algebras of $(\bbV,\circ)$ such that each $(\bbV_i,\circ)$ has 
identity element $e_i\in \bbV_i$, for $i\in[n]$. 
We say that $(\bbV,\circ)$ is the direct sum of $\{(\bbV_i,\circ)\}_{i=1}^n$ if $\bbV$ is the direct sum of the subspaces $\{\bbV_i\}_{i=1}^n$ (denoted by $\bbV = \bigoplus_{i=1}^n \bbV_i$) and $\bbV_i\circ \bbV_j = \{0\}$ for any $i\ne j$, $i,j\in[n]$. 
From~\cite[Proposition III.4.4]{Faraut_94}, we know that any Euclidean Jordan algebra can be written as the direct sum of 
simple Euclidean Jordan sub-algebras in a {\em unique} way. 

Let $(\bbV,\ipt{\cdot}{\cdot})$ be a real inner-product space. 
We call a symmetric cone $\calK\subseteq\bbV$ {\em irreducible} if there do not exist non-trivial subspaces $\bbV_1$ and $\bbV_2$ of $\bbV$ such that $\bbV = \bbV_1\bigoplus \bbV_2$ and symmetric cones $\calK_1\subseteq\bbV_1$ and $\calK_2\subseteq\bbV_2$ such that $\calK = \calK_1 + \calK_2$. A symmetric cone can be written as the direct sum of irreducible symmetric cones in a unique way. In fact, a symmetric cone is irreducible if and only if it is the cone of squares of some simple Euclidean Jordan algebra, and there exists a one-to-one correspondence between irreducible symmetric cones and simple Euclidean Jordan algebras. 

From~\cite[Chapter~V]{Faraut_94}, the simple Euclidean Jordan algebras (and hence irreducible symmetric cones) can be completely classified into the five kinds in the following:

\begin{enumerate}[label=\roman*)]
\item The Jordan spin algebra $\bbL^{n+1}$: $\bbV=\bbR^{n+1}$ and for $x\in \bbR^{n+1}$, denote $x:= (x_0,\barx)\in \bbR\times \bbR^n$. For any $x,y\in \bbR^{n+1}$, $x\circ y:= (x^\top y, x_0\bary+y_0\barx)$, $\tr(x)= 2x_0$ and $\ipt{x}{y} = 2x^\top y$. 
The identity element $e= (1,0)\in \bbR\times \bbR^n$ and the cone of squares $\calK(\bbV):= \{x\in\bbR^{n+1}: x_0\ge \normt{\barx} \}$, namely the second-order cone. Lastly, $\rank(\bbV)=2$.  

\item $\mathrm{Herm}(n,\bbF)$, where $\bbF = \bbR$, $\bbC$ or $\bbH$ (and $\bbH$ denotes quaternion): $\bbV$ 
 is the space  of $n\times n$ Hermitian matrices with entries in $\bbF$.  For any $X,Y\in\bbV$, 
 $X\circ Y := (XY+YX)/2$, $\tr(X)$ is defined in the usual matrix sense and $\ipt{X}{Y}: =\tr(XY)$. 
The identity element $e:= I_n$ and $\calK(\bbV) = \calH_+^n(\bbF)$,  i.e., the cone of $n\times n$ Hermitian PSD matrices with entries in $\bbF$. Lastly, $\rank(\bbV)=n$.

\item $\mathrm{Herm}(3,\bbO)$ (where $\bbO$ denotes octonion): $\bbV$ is the space of $3\times 3$  Hermitian matrices with entries in $\bbO$. The {bilinear} product $\circ$, $\tr(\cdot)$, inner product $\ipt{\cdot}{\cdot}$ 
and identity element are defined in the same way as in $\mathrm{Herm}(n,\bbF)$, and  $\calK(\bbV)= \calH_+^3(\bbO)$. Lastly, $\rank(\bbV)=3$.  
\end{enumerate}

\noindent 
Besides the five simple Euclidean Jordan algebras above, another Euclidean Jordan algebra that we will frequently encounter in this work is $(\bbR^n, \circ)$, 
where for $x,y\in \bbR^n$, $x\circ y = (x_iy_i)_{i=1}^n$, $\tr(x) = \sum_{i=1}^n x_i$, 
  $\ipt{x}{y}= \sum_{i=1}^n x_iy_i$, and the identity $e = (1,1,\ldots,1)$.  Also, for $x\in \bbR^n$, its eigenvalues are $\{x_i\}_{i=1}^n$ and its Jordan frame  is 
  $\{e_i\}_{i=1}^n$. (Note that all $x\in \bbR^n$ have the same Jordan frame.) 
Indeed, this Euclidean Jordan algebra can be viewed as the Cartesian product of $n$ copies of $\mathrm{Herm}(1,\bbR)$, and hence we denote it by $\mathrm{Herm}(1,\bbR)^n$. 



\section{Some Convex Analytic Properties of Legendre and LH Functions} \label{sec:legendre_LH}

In this section, we shall establish some convex analytic properties of Legendre and LH functions, which not only play important roles in the development and analysis of our GMG method, but also may be independent interest. 
To start, let us present some very nice properties of Legendre functions. 
Given a proper, closed and convex function $h:\bbY\to \barbbR$, define its Fenchel conjugate $h^*(z) := {\sup}_{y\in\bbY} \; \ipt{z}{y} - h(y)  $, where $z\in\bbY^*$.

\begin{lemma}[{\cite[Theorem~26.5]{Rock_70}}]\label{lem:legendre}
Let  $h:\bbY\to \barbbR$ be a proper, closed and convex function. Then $h$ is Legendre if and only if $h^*$ 
is Legendre. Furthermore, $\nabla h:\inter\dom h\to \inter\dom h^*$ is a homeomorphism, whose inverse $(\nabla h)^{-1} = \nabla h^*$. 
\end{lemma}

\noindent
We also need the following simple lemma in our proof. 

\begin{lemma} \label{lem:simplex}
Let $\calA\subseteq\bbV$ be a nonempty convex set such that $\Omega:= \cl \calA$ is a closed, convex, pointed cone that strictly contains $0$ (i.e., $\{0\}\subsetneq \Omega$). Let $s\in \inter \Omega^*$ and define $\calQ_s:= \{x\in\calA:\ipt{s}{x}=1\}$. Then we have 
\begin{enumerate}[leftmargin = 5ex,label=(\roman*)]
\item \label{item:ri} $\ri\calQ_s= \{x\in\ri\Omega: \ipt{s}{x}=1\}$, and both $\ri\calQ_s$ and $\calQ_s$ are nonempty, convex  and bounded;
\item \label{item:cl} $\cl\calQ_s= \{x\in\Omega:\ipt{s}{x}=1\}$, which  is nonempty, convex  and compact. 
\end{enumerate}
\end{lemma}

\begin{proof}
See Appendix~\ref{app:proof_simplex}.  
\end{proof}

\noindent
Define $\Gamma:=\cl\dom h$. The convex analytic properties of Legendre and LH functions are shown in the following theorem. 

\begin{theorem} \label{thm:Legendre_LH}
Let $h:\bbY\to \barbbR$ be Legendre and $\theta$-LH (cf.~Definitions~\ref{def:legendre} and~\ref{def:LH}). Then

\begin{enumerate}[leftmargin = 4ex,label=(\roman*)]
\item \label{item:pos_scaling} $\dom h$ is invariant under positive scaling (i.e., $t \,\dom h = \dom h$ for all $t>0$), and $0\not\in\dom h$. 
\item \label{item:Gamma} $\Gamma$ is a closed, convex and solid cone, and $\inter\Gamma = \inter \dom h$ is an open convex cone. 
\item \label{item:h^*} $h^*$ is Legendre and  $\theta$-LH, 
$\cl\dom h^* = \Gamma^\circ$ and hence 
$\inter\dom h^* = \inter\Gamma^\circ$. In addition, both $\Gamma$ and $\Gamma^\circ$ are regular cones.  
\item \label{item:theta} For any $y\in \inter\dom h$, we have  $\ipt{\nabla h(y)}{y}=-\theta$ and hence $h^*(\nabla h(y)) + h(y) = -\theta$. 
\item \label{item:curvature} For all $y\in \dom h$ and $s\in\inter\dom h^*$, we have 
\begin{equation}
\Phi(y,s):= h(y) + h^*(s) + \theta\ln(-\ipt{s}{y})\ge \theta \ln \theta - \theta.   \label{eq:Phi_lb} 
\end{equation}
As a result, for all $y'\in \inter\dom h$ and $y\in \dom h$, we have
\begin{equation}
h(y') - h(y)\le \theta\ln\big({-\ipt{\nabla h(y')}{y}}/{\theta}\big). \label{eq:curv_ub}
\end{equation}
\end{enumerate}
\end{theorem}

\begin{proof}
{\em Proof of Part~\ref{item:pos_scaling}}:   From Definition~\ref{def:LH}, if $y\in \dom h$, then for all $t>0$,  $ty\in \dom h$, so $t \,\dom h \subseteq \dom h$ for all $t>0$.  Hence $t^{-1} \,\dom h \subseteq \dom h$ for all $t>0$, and we have $t \,\dom h = \dom h$ for all $t>0$. Now  suppose that $0\in \dom h$, then we have $h(0) = h(0) - \theta\ln t$ for all $t>0$, 
which is a contradiction.\\[-2ex] 

\noindent 
{\em Proof of Part~\ref{item:Gamma}}:
 note that by Definitions~\ref{def:legendre} and~\ref{def:LH}, $\inter\dom h\ne \emptyset$ is an open convex cone.
  Therefore, $\Gamma=\cl \dom h = \cl( \inter\dom h)$ is a  solid closed convex cone, 
  and $\inter\Gamma = \inter\dom h$.  
  To prove Part~\ref{item:h^*}, note that the Legendreness of $h^*$ follows from Lemma~\ref{lem:legendre}. Also, for any $t>0$, 
\begin{align}
h^*(ts) = \sup_{y\in\dom h} \, t\ipt{s}{y} - h(y) &= \sup_{y\in\dom h} \, \ipt{s}{ty} - h(ty) - \theta \ln t \label{eq:h^*_1}\\
& \eqa \sup_{u\in\dom h} \, \ipt{s}{u} - h(u) - \theta \ln t = h^*(s) - \theta \ln t, 
\end{align}
where in (a) we use Part~\ref{item:pos_scaling}. This shows that $h^*$ is $\theta$-LH. 
Next, we show that $\dom h^* \subseteq \Gamma^\circ$.   Since $\Gamma^\circ =\{s\in \bbY^*:\ipt{s}{y}\le 0, \,\forall\,y\in\inter\Gamma\}$, for any $s\not\in \Gamma^\circ$, 
there exists $y\in \inter\Gamma$ 
such that $\ipt{s}{y}>0$, and since $\inter\Gamma$ is an open convex cone, we have $\{ty:t>0\}\subseteq \inter\Gamma = \inter \dom h$. By definition, 
\begin{align*}
h^*(s)\ge {\sup}_{t>0}\;  t\ipt{s}{y} - h(ty) = {\sup}_{t>0}\; t\ipt{s}{y} + \theta\ln t - h(y) = +\infty, 
\end{align*}
and 
hence $s\not\in \dom h^*$. This shows $\dom h^* \subseteq \Gamma^\circ$, and hence $\inter \dom h^*\subseteq\cl\dom h^* \subseteq \Gamma^\circ$. Since $h^*$ is Legendre, $\inter \dom h^*\ne \emptyset$, and hence $\Gamma^\circ$ is solid. Since $\Gamma$ is closed and convex, $\Gamma$ is pointed (resp.~solid) if and only if $\Gamma^\circ$ is solid (resp.~pointed). This shows that both $\Gamma$ and $\Gamma^\circ$ are regular. \\[-2ex]


\noindent 
{\em Proof of Part~\ref{item:h^*}}:
Next, we show that $h^*$ is differentiable on $\inter \Gamma^\circ$. By the Legendreness of $h^*$, 
this implies that $\inter \Gamma^\circ\subseteq \inter\dom h^*$, and hence $\Gamma^\circ\subseteq \cl \dom h^*$.  Take any $s\in \inter \Gamma^\circ$, and define $\calQ_s :=  \{y\in\dom h: \ipt{s}{y}=-1\}$. Note that  $\dom h = \{ty:y\in \calQ_s,t>0\}$, which easily follows from Part~\ref{item:pos_scaling}. 
Indeed, since $0\not\in\dom h$ and $\dom h\subseteq \Gamma$, for any $y\in\dom h$, we have $\ipt{s}{y}<0$, and hence 
$\bary:= y/(-\ipt{s}{y})\in \calQ_s$. 
As a result, $y = (-\ipt{s}{y}) \bary\in \{ty:y\in \calQ_s,t>0\}$. 
 Due to this, we have 
\begin{align}
h^*(s) = \sup_{y\in\dom h} \, \ipt{s}{y} - h(y) &= \sup_{y\in\calQ_s}\sup_{t>0} \; t\ipt{s}{y} - h(ty), \label{eq:h^*_s0} 
\\
&= \sup_{y\in\calQ_s}\sup_{t>0} \; -t - h(y) + \theta \ln t=   \theta \ln \theta - \theta -\inf_{y\in\calQ_s}\; h(y) , \label{eq:h^*_s}
\end{align}
Define $\calY^*:= \argmin \{h(y): y\in\calQ_s\},$
and note that $\partial h^*(s) =\theta \calY^*$. 
Since $\Gamma$ is regular, by Lemma~\ref{lem:simplex}, we know that $\cl \calQ_s=\{y\in\Gamma: \ipt{s}{y}=-1\}$ is a nonempty, convex and compact set. 
Note that since $\inter\dom h\cap \calQ_s=\inter\Gamma\cap \calQ_s\ne \emptyset$ and $h$ is essentially smooth, 
we know that (i) $\calY^*\subseteq\inter\dom h$
and (ii) $\calY^* =  \argmin \{h(y): y\in \cl\calQ_s\}$. 
Since $h$ is proper and closed and $\cl\calQ_s$ is compact, we know that $\calY^*\ne\emptyset$. Since $\calY^*\subseteq\inter\dom h$,  by the essential strict convexity of $h$, $\calY^*$ is a singleton. As a result, $\partial h^*(s)$ is a singleton and hence $h^*$ is differentiable at $s\in \inter \Gamma^\circ$.\\[-2ex] 

\noindent 
{\em Proof of Part~\ref{item:theta}}:
Note that by differentiating both sides of~\eqref{eq:LH} w.r.t.\ $t$ and setting $t=1$, we have $\ipt{\nabla h(y)}{y}=-\theta$ for all $y\in \inter\dom h$. The second identity follows from~\cite[Theorem~23.5]{Rock_70}. \\[-2ex] 


\noindent 
{\em Proof of Part~\ref{item:curvature}}:
Fix any  $y\in \dom h\subseteq \Gamma$ and $s\in\inter\Gamma^\circ$, and let $\bary:= y/(-\ipt{s}{y})\in \calQ_s$  be given in the proof of Part~\ref{item:h^*}. 
By~\eqref{eq:h^*_s}, we have 
\begin{equation}
h^*(s) \ge  \theta \ln \theta - \theta- h(\bary) =  \theta \ln \theta - \theta- h(y) - \theta \ln(-\ipt{s}{y}).
\end{equation} 
This proves~\eqref{eq:Phi_lb}. 
Now, for any $y'\in \inter\dom h$, if we set $s = \nabla h(y')\in \inter\dom h^*$ in~\eqref{eq:Phi_lb}, then by using Part~\ref{item:theta} and rearranging, we obtain~\eqref{eq:curv_ub}. 
\end{proof}

\begin{corollary} \label{cor:error_bd}
Let $h:\bbY\to \barbbR$ be Legendre and $\theta$-LH, and $\calY$ be a compact set such that $\calY\cap\dom h\ne \emptyset$. Define $\calY^*:= \argmin_{y\in\calY}\, h(y)$. Then $\calY^*\ne \emptyset$, 
and for any $y^*\in \calY^*$, 
\begin{equation}
h(y) - h(y^*)\le \theta\ln\big({-\ipt{\nabla h(y)}{y^*}}/{\theta}\big), \quad \forall\,  y\in \inter\dom h \ . 
\end{equation}
\end{corollary}

\begin{proof}
Since $\calY\cap\dom h\ne \emptyset$, $h$ is proper and closed, and $\calY$ is compact, we know that $\calY^*\ne \emptyset$ and $\calY^*\subseteq\dom h$. 
We then invoke~\eqref{eq:curv_ub} in Theorem~\ref{thm:Legendre_LH} to finish the proof. 
\end{proof}

%

%
%

\section{The GMG Method} \label{sec:algo}

For the convenience of subsequent analysis, let us introduce a few notations. First, we let $(\bbV,\circ)$ 
 be the Euclidean Jordan algebra associated with $\calK$ in~\eqref{eq:P}, and denote its identity element by $e$ and its rank by $n$. 
Second, we  denote the set of optimal solution of~\eqref{eq:P}  by $\calX^*$. 
Third, define 
\begin{equation}
\barF := -F,\quad  \barf := -f \quad\andd\quad \Gamma:=\cl\dom \barf,  \label{eq:def_-F}
\end{equation}
and since $\barf$ 
 is  Legendre and LH, by Theorem~\ref{thm:Legendre_LH}\ref{item:h^*}, both $\Gamma$ and $\Gamma^*$ are  regular cones. 
Finally, to avoid triviality, 
we assume that both vector spaces $\bbV$ and $\bbY$ are nontrivial. 

\subsection{Additional Assumptions and Their Implications}

\begin{assump}\label{assump:A}
$\rvA:\bbV\to \bbY$ is a linear operator that satisfies 
$\rvA(\inter\calK)\subseteq \inter\Gamma$, and its adjoint $\rvA^*:\bbY^*\to \bbV$ satisfies that $\rvA^*(\inter\Gamma^*)\subseteq \inter\calK$. 
\end{assump}

\noindent 
The implications of Assumption~\ref{assump:A} are shown in Proposition~\ref{prop:well_posed} below. In particular,   Assumption~\ref{assump:A} ensures that~\eqref{eq:P} is well-posed. As we shall see later in Section~\ref{sec:algo}, Assumption~\ref{assump:A} also ensures that our GMG method is well-defined.  

\begin{prop} \label{prop:well_posed}
Under Assumption~\ref{assump:A}, \eqref{eq:P} is feasible 
and $\calX^*\ne \emptyset$.   
In addition,  the function $\barF:\bbV\to \barbbR$ is proper, closed, convex, and  differentiable on $\inter\calK$.  Moreover, $\nabla f: \inter \Gamma \mapsto \inter \Gamma^*$, and for all $x\in \inter\calK$, $\nabla F(x)\in \inter\calK$ and $\ipt{\nabla F(x)}{x} = \theta$. 
\end{prop}

\begin{proof}
Let $x=(1/n)e$, 
then $\tr(x) = 1$ and $x\in \inter\calK$, and hence $x\in\calC$. 
Under Assumption~\ref{assump:A}, we know that $\rvA x\in \inter\Gamma = \inter\dom \barf$, and hence $x\in \dom \barF $. Since $\barf$ is proper, closed, convex, this implies that $\barF$ is also proper, closed, convex. Since $e\in \inter\calK$ and $\calK$ is regular, by Lemma~\ref{lem:simplex}, we know that $\calC\ne\emptyset$ is convex and compact. Since $\dom \barF \cap \calC\ne\emptyset$, 
we know that $\calX^*\ne \emptyset$. Finally, for all $x\in \inter\calK$, $\rvA x\in  \inter \Gamma = \inter\dom \barf$, on which $\barf$ is differentiable (since it is Legendre). 
This implies that $F$ is differentiable on $\inter\calK$.  By Theorem~\ref{thm:Legendre_LH}\ref{item:h^*}, we know that $\nabla \barf: \inter \Gamma \mapsto  \inter\Gamma^\circ$, and hence 
$\nabla f: \inter \Gamma \mapsto \inter \Gamma^*$. Thus for all $x\in \inter\calK$, $\nabla F(x)= \rvA^*\nabla f(\rvA x) \in \rvA^*(\inter\Gamma^*)\subseteq \inter\calK$ (by Assumption~\ref{assump:A}). In addition, by Theorem~\ref{thm:Legendre_LH}\ref{item:theta}, we have  
$\ipt{\nabla F(x)}{x} = \ipt{\rvA^*\nabla f(\rvA x)}{x} = \ipt{\nabla f(\rvA x)}{\rvA x} = \theta$. 
\end{proof}

\noindent 
Under Assumption~\ref{assump:A}, we can present our second assumption. Note that unlike Assumption~\ref{assump:A}, this assumption is only used in the convergence-rate analysis of the GMG method. 

\begin{assump}
\label{assump:F}
The log-gradient $\ln(\nabla F(\cdot))$ 
is $\calK$-convex on  $\inter\calK$, i.e., for any $x,y\in\inter\calK,$ 
\begin{align}
\ln(\nabla F(\lambda x + (1-\lambda)y)) \preceq_{\calK}\, \lambda\ln(\nabla F(x)) + (1-\lambda)\ln(\nabla F(y)), \; \forall\,\lambda\in[0,1]. 
\label{eq:log_grad_convex}  
\end{align} 
\end{assump}

\noindent 
At first glance, Assumption~\ref{assump:F} seems quite unconventional. 
 To our knowledge, 
this assumption has not appeared in the analysis of any first-order or higher-order methods (except in some special cases), and one may naturally wonder if this assumption is mild enough to be satisfied by some meaningful function classes. 
In Section~\ref{sec:function}, we will study this assumption in details, and identify several classes of functions that satisfy  Assumption~\ref{assump:F}. 

\subsection{The GMG Method}\label{sec:algo}

\begin{algorithm}[t!]
\caption{The Generalized Multiplicative Gradient (GMG) Method}\label{algo:GMG}
\begin{algorithmic}
\State {\bf Input}: starting point $x^0\in \ri\calC$, exponent parameter $\alpha\in(0,1]$ \phantom{r} \vspace{1ex}
\State {\bf At iteration $t\ge 0$}:
\begin{align}
\hatx^{t+1}&:= \exp\big\{\ln\big(x^t\big)+\ln\big(\nabla F(x^t)^\alpha\big)\big\}\label{eq:GMG1} \\
x^{t+1}&:= { \hatx^{t+1}}/{\tr(\hatx^{t+1})} \label{eq:GMG2} 
\end{align}
\end{algorithmic}\vspace{-.2cm}
\end{algorithm}


Our GMG method is shown in Algorithm~\ref{algo:GMG}, 
which has an extremely simple structure. 
Note that the functions $\exp(\cdot)$, $\ln(\cdot)$ and $(\cdot)^\alpha$ in~\eqref{eq:GMG1} are applied to $x\in\bbV$ in the sense of~\eqref{eq:func_x}. 
 Let us make a few remarks about  Algorithm~\ref{algo:GMG}. First, note that under Assumption~\ref{assump:A}, Algorithm~\ref{algo:GMG} is indeed well-defined. From Lemma~\ref{lem:simplex}, we know that 
\begin{equation}
\ri\calC= \calH \cap\inter\calK, \quad \where \;\;\;  \calH: =\{x\in\bbV: \tr(x)=1\}. 
\end{equation} 
 Since $x^0\in \inter\calK$ and $\nabla F(x)\in \inter\calK$ for all $x\in \inter\calK$ (cf.\ Proposition~\ref{prop:well_posed}), using induction, we see that at each iteration $t\ge 0$, $x^t\in \ri\calC$  and all the operations in~\eqref{eq:GMG1} and~\eqref{eq:GMG2} are well-defined. 

 Second, 
 with the exponent parameter $\alpha=1$, Algorithm~\ref{algo:GMG} recovers the MG method~\eqref{eq:MG} and its variant~\eqref{eq:MG_matrix} for solving the specific problem instances~\eqref{eq:PET},~\eqref{eq:DOPT} and~\eqref{eq:QST}. In particular, when $\calK = \bbR_+^n$, the Euclidean Jordan algebra associated with $\calK$ is $\mathrm{Herm}(1,\bbR)^n$ (cf.\ last part of Section~\ref{sec:background}).  
  In this case,~\eqref{eq:GMG1} and~\eqref{eq:GMG2} in Algorithm~\ref{algo:GMG} simplifies to 
 \begin{align}
\hatx^{t+1}:= x^t\odot \nabla F(x^t)^\alpha, \qquad x^{t+1}:= { \hatx^{t+1}}/\ipt{e}{\hatx^{t+1}}. \label{eq:GMG_R^n}
\end{align} 
When $\alpha=1$ and $\theta=1$, by Proposition~\ref{prop:well_posed}, we have $\ipt{e}{\hatx^{t+1}} = \ipt{x^t}{\nabla F(x^t)} = 1$, and the second step in~\eqref{eq:GMG_R^n} can be omitted. Hence Algorithm~\ref{algo:GMG} reduces to the MG method in~\eqref{eq:MG}. 

Third, the additional exponent parameter $\alpha\in(0,1]$ adds more flexibility to the GMG method, and the inspiration is drawn from the literature of the MG method for  optimal design problems under the general optimality criteria (see e.g.,~\cite{Silvey_78,Yu_10}). 
 In particular, for the D-optimal design problem~\eqref{eq:DOPT}, 
 Yu~\cite[Corollary~1]{Yu_10} showed that for all $\alpha\in(0,1]$, the iterates $\{x^t\}_{t\ge 0}$ generated by Algorithm~\ref{algo:GMG}  converge to the set of 
 optimal solutions   (see also~\cite[Corollary 2.1]{Lu_13}).  
On the other hand, for~\eqref{eq:DOPT}, there largely lacks a  theoretical understanding of the dependence of the convergence rate of Algorithm~\ref{algo:GMG} on $\alpha$, and due to this, the ``optimal'' $\alpha$ is usually chosen on a ad-hoc basis (see e.g,.~\cite{Silvey_78}). Interestingly, our convergence rate analysis of Algorithm~\ref{algo:GMG} (cf.~Theorem~\ref{thm:rate}) offers new insight into this issue. Specifically, our analysis suggests that the optimal choice of $\alpha$ in this case should be 1 (cf.~Remark~\ref{rmk:optimal_choice}).  

 

\subsection{Convergence Rate Analysis} 

To analyze the convergence rate of Algorithm~\ref{algo:GMG}, we will need a few  lemmas. The first one below is the Golden-Thompson (GT) inequality in Euclidean Jordan algebra~\cite[Theorem 5.1]{Tao_21}, which is an extension of the GT inequality for complex Hermitian matrices~\cite{Golden_65,Thompson_65} . 

\begin{lemma}[{\cite[Theorem 5.1]{Tao_21}}] \label{lem:GT}
For any $x,y\in\bbV$, we have 
\begin{equation}
\tr(\exp(x+y))\le \tr(\exp(x)\circ \exp(y)).
\end{equation}
\end{lemma}

\noindent 
The next three lemmas are quite simple, and their proofs are deferred to Appendix~\ref{app:proof_analysis}. Notation-wise, 
for $x\in\bbV$, denote its eigenvalues by $\lambda_1(x)\ge \cdots \ge \lambda_n(x)$, and   define $\lambda_{\max}(x):= \lambda_1(x)$ and $\lambda_{\min}(x):= \lambda_n(x)$. Also, let  $\succeq $ denote the partial order $\succeq_\calK $.

\begin{lemma} \label{lem:lambda_max}
For all $x,y\in\bbV$, $\lambda_{\max}(x) = \max_{z\in\calC}\, \ipt{z}{x} $. If $x\preceq y$, then $\lambda_{\max}(x)\le \lambda_{\max}(y)$. 
\end{lemma}

\begin{lemma} \label{lem:alpha}
For any $x\in\calC$, $y\in\inter\calK$ and $\alpha\in(0,1]$, we have $ \tr({x} \circ {y^\alpha}) = \ipt{x}{y^\alpha}\le \ipt{x}{y}^\alpha. $
\end{lemma}

\begin{lemma} \label{lem:error_bd}
Under Assumption~\ref{assump:A}, for any $x\in\inter\calK$, we have 
\begin{equation}
F^* - F(x)\le \theta\lambda_{\max}\big(\ln\left({\nabla F(x)}/{\theta}\right)\big). 
\end{equation}
\end{lemma}

\noindent 
Equipped with the lemmas above, we are ready to present the convergence rate of Algorithm~\ref{algo:GMG}.

\begin{theorem} \label{thm:rate}
Under Assumptions~\ref{assump:A} and~\ref{assump:F}, 
in Algorithm~\ref{algo:GMG}, we have 
\begin{equation}
F^* - F(\barx^t)\le \frac{\theta\ln(1/\lambda_{\min}(x^0))}{\alpha(t+1)}, \quad{\rm where} \;\; \barx^t:=\tfrac{1}{t+1}\textstyle 
{\sum}_{k=0}^t \;x^k,\quad \forall\, t\ge 0. \label{eq:rate_GMG}
\end{equation}
\end{theorem}

\begin{proof}
Since $x^t\in \ri\calC$ for $t\ge 0$, by Lemmas~\ref{lem:GT} and~\ref{lem:alpha} and Proposition~\ref{prop:well_posed}, we have  
\begin{align}
\tr(\hatx^{t+1}) = \tr\big(\exp\big(\ln(x^t)+\ln(\nabla F(x^t)^\alpha)\big)\big)\le \tr\big(x^t\circ \nabla F(x^t)^\alpha\big) \le \ipt{x^t}{\nabla F(x^t)}^\alpha = \theta^\alpha.  
\end{align}
As a result, 
we have for all $t\ge 0$, 
\begin{align}
\ln(x^{t+1}) &= \ln(\hatx^{t+1}) - \ln(\tr(\hatx^{t+1}))e\\ 
& = \ln(x^{t}) + \ln(\nabla F(x^t)^\alpha)- \ln(\tr(\hatx^{t+1}))e\\
& \succeq \ln(x^{t}) + \alpha\ln(\nabla F(x^t))- \alpha \ln(\theta) e. \label{eq:before_telescope}
\end{align}
Telescoping~\eqref{eq:before_telescope} over $k=0,\ldots,t$ and using Assumption~\ref{assump:F}, we have
\begin{align}
\ln(x^{t+1}) &\succeq \ln(x^{0}) + \textstyle  \alpha\sum_{k=0}^{t} \, \ln(\nabla F(x^k))- \alpha(t+1) \ln(\theta) e\\
&\succeq \ln(x^{0}) + \alpha(t+1)\ln(\nabla F(\barx^t))- \alpha (t+1) \ln(\theta) e\\
&= \ln(x^{0}) + \alpha(t+1)\ln(\nabla F(\barx^t)/\theta)
\end{align}
Since $x^{t+1}\in\ri\calC$, we have $\lambda_i(x^{t+1})\in(0,1)$ for all $i\in[n]$, and hence $\ln(x^{t+1})\preceq 0$. Thus 
\begin{align}
\ln(\nabla F(\barx^t)/\theta) \preceq -\ln(x^0)/(\alpha (t+1)). 
\end{align} 
Now, by Lemma~\ref{lem:lambda_max}, 
we have
\begin{align}
\lambda_{\max}\big(\ln(\nabla F(\barx^t)/\theta)\big)&\le \lambda_{\max}\big(-\ln(x^0)/(\alpha (t+1))\big)\\
& = -\lambda_{\min}\big(\ln(x^0)\big)/(\alpha (t+1)) = -\ln\big(\lambda_{\min}(x^0)\big)/(\alpha (t+1))
\end{align}
This,  together with Lemma~\ref{lem:error_bd}, completes the proof. 
\end{proof}

\begin{remark} \label{rmk:optimal_choice}

Note that Theorem~\ref{thm:rate} provides us with some insight in choosing the starting point $x^0\in \ri\calC$ and the exponent parameter $\alpha\in(0,1]$. Specifically, 
if we choose $x^0$ and $\alpha$ to minimize 
the upper bound of the convergence rate in~\eqref{eq:rate_GMG}, 
then the optimal choices are $x^0 = (1/n) e$ and $\alpha = 1$. Under such choices, the convergence rate of Algorithm~\ref{algo:GMG} becomes 
\begin{equation}
F^* - F(\barx^t)\le \frac{\theta\ln(n)}{t+1}, \quad \forall\, t\ge 0. \label{eq:opt_conv_rate}
\end{equation}
\end{remark}

\begin{remark} \label{rmk:x^0_rate}
Note that the convergence rate in~\eqref{eq:rate_GMG} also applies to $\barx^0 = x^0$, which rarely occurs for FOMs. (Indeed, since the starting point $x^0$ is chosen before running the algorithm, we typically can only provide computational guarantees  for the first (averaged) iterate onward.) 
Specifically, since $x^0$ can be any point in $\ri\calC$, 
 we have 
\begin{equation}
\delta(x):=F^* - F(x)\le {\theta\ln(1/\lambda_{\min}(x))}, \quad \forall\, x\in\ri\calC. \label{eq:delta(x)}
\end{equation}
In particular, let $x_\rmc := (1/n)e$, then we have  $\delta(x_\rmc)\le \theta\ln(n)$. 
In fact, the same bound on $\delta(x_\rmc)$ 
has been derived in~\cite[Lemma~7]{Zhao_23pet} and~\cite[Lemma~3]{Khachiyan_96} 
for the two specific problems~\eqref{eq:PET} and~\eqref{eq:DOPT}, respectively, in very different ways.  
In view of this, the bound in~\eqref{eq:delta(x)} 
extends these 
results not only to the whole problem class~\eqref{eq:P}, but also to 
any point $x\in \ri\calC$. As such, \eqref{eq:delta(x)}  can be viewed as an interesting byproduct of the convergence analysis of Algorithm~\ref{algo:GMG}.
 
\end{remark}

\begin{remark}
Note that for solving~\eqref{eq:P}, we can always scale 
$f$ (or equivalently, $F$) by a factor of $\theta^{-1}$, so that $f$ becomes $1$-LH.  
In this case, the convergence rate  in~\eqref{eq:rate_GMG} only involves two quantities, namely  $\alpha$ and $\lambda_{\min}(x^0)$, both of which only depend on the input to Algorithm~\ref{algo:GMG}, and hence are known to us before we run the algorithm. In view of this, the convergence rate  in~\eqref{eq:rate_GMG} is {\em problem-data independent}, and can be used as a stopping criterion for Algorithm~\ref{algo:GMG}. 



\end{remark}




\section{Function Classes Satisfying Assumptions~\ref{assump:A} and~\ref{assump:F}} \label{sec:function}

Let us start by providing an equivalent statement of Assumption~\ref{assump:F}. To that end, let $\PI(\bbV)$ denote the set of primitive idempotents in $\bbV$, namely  $\PI(\bbV) = \{q\in\bbV:q^2 = q, \tr(q) = 1\}\subseteq \calC$.

\begin{prop} \label{prop:LGrad_convex}
Under Assumption~\ref{assump:A}, 
$F$  satisfies Assumption~\ref{assump:F} 
if and only if 
 the 
function 
\begin{equation}
\phi_u:x\mapsto \ipt{\ln(\nabla F(x))}{u} 
\label{eq:phi_u}
\end{equation}
is convex on $\inter\calK$, for any $u\in\PI(\bbV)$.
\end{prop}

\begin{proof}
See Appendix~\ref{app:proof_func}. 
\end{proof}


\noindent 
In the following, we will introduce two scenarios where $f$ and $\rvA$ satisfy Assumptions~\ref{assump:A} and~\ref{assump:F}. Let us start with the first scenario. 

\subsection{First Scenario
}

As mentioned in Section~\ref{sec:algo}, the Euclidean Jordan algebra associated with $\bbR_+^n$ 
is $\mathrm{Herm}(1,\bbR)^n$, and $\PI(\bbR^n)=  \{e_i\}_{i=1}^n$. Also, in this case, $\calC = \Delta_n$. 
To introduce $f$ and $\rvA$, we need to first introduce some background on hyperbolic polynomials and hyperbolicity cones. 

%
%
%
%
%
%
\subsubsection{Background on hyperbolic polynomials and hyperbolicity cones} \label{sec:hyperbolic}
\noindent 
Let $p:\bbY\to \bbR$  be a real homogeneous polynomial of degree  $m\ge 1$  
and the direction $d \in \bbY$  satisfies that  $p(d) > 0$ . Then  $p$  is {\em hyperbolic} 
w.r.t.\  $d$, if for every  $y \in \bbY$, the (univariate) 
polynomial  $t \mapsto p(y - td)$  
has only real roots, which are denoted by $\lambda_1^d(y)\ge \cdots\ge  \lambda_m^d(y)$. 
Notation-wise, 
this is written as $p\in \calH_m(d)$, where $\calH_m(d)$ denotes the set of hyperbolic polynomial on $\bbY$ of degree  $m$ w.r.t.\ the direction $d\in\bbY$. 
Define $\lambda^d(y):= (\lambda_1^d(y), \ldots,\lambda_m^d(y))$ for $y \in \bbY$. We call $\lambda^d:\bbY\to \bbR^m$ 
 the {\em spectral map} w.r.t.\ $p$ and $d$, and 
 $\lambda_k^d(y)$ the $k$-th eigenvalue of $y$ w.r.t.\ $p$ and $d$, where $1\le k\le m$. (Note that for notational brevity, the dependence of $p$ in $\lambda^d$ and   $\lambda^d_k$ has been suppressed, and the same applies to other notations in the following.) 

 Next, let us define the hyperbolicity cone of $p$ w.r.t.\ $d$ as 
 \begin{equation}
 \calC(d):=\{y\in\bbY: \lambda^d_m (y)>0\}. 
 \end{equation}
 From the literature~\cite{Garding_59,Bauschke_01a,Renegar_06}, it is well-known 
 that   $\calC(d)$ is the connected component of $\{y\in\bbY: p (y)\ne 0\}$ that contains $d$, and for any $a\in \calC(d)$, $p$ is hyperbolic w.r.t.\  $a$   and $\calC(a) = \calC(d)$. 
   Furthermore,  $\calC(d)$ is an open convex cone and  $\cl\calC(d) = \{y\in\bbY: \lambda^d_m (y)\ge 0\}$ is a closed, convex and solid cone. Clearly, we have that $p(y)>0$ for $y\in \calC(d)$ and $p(y)=0$ for  $y\in\bdry\calC(d)$. Finally, we call $p$ {\em complete} if $\calZ_p(d):=\{y\in\bbY:\lambda^d(y) = 0\} = \{0\}$.

Note that the class of complete hyperbolic polynomials is fairly broad. For example, from Section~\ref{sec:background}, we know that given any Euclidean Jordan algebra $(\bbY,\circ)$ 
with identity $e$ and rank $m\ge 1$,  the   
function 
$\det:\bbY\to\bbR$ (defined in~\eqref{eq:def_det}) is a complete hyperbolic polynomial of degree $m$ w.r.t.\ the direction $e$. Furthermore, the hyperbolicity cone $\calC(e) = \inter \calK(\bbY)$, where $\calK(\bbY)$ is 
the cone of squares of $\bbY$. For more examples, we refer readers to~\cite[Section~6]{Bauschke_01a}.



\subsubsection{Definitions of $f$ and $\rvA$} 
Based on the introduction in Section~\ref{sec:hyperbolic}, 
we define 
\begin{equation}
f(y):= \left\{\begin{array}{ll}
\ln p(y), \quad &\forall \, y\in \calC(d)\\
-\infty,  \quad &\forall \, y\in \bbY\setminus \calC(d)
\end{array}\right. \quad \mbox{and} 
\qquad \barf := -f,  \label{eq:def_case_i}
\end{equation}
where $p\in\calH_m(d)$ is  {\em complete} with hyperbolicity cone $\calC(d)$. 
Clearly, $\dom \barf = \calC(d)$ and hence $\Gamma:=\cl\dom \barf = \cl \calC(d)$.  In addition, the linear operator $\rvA:\bbR^n \to \bbY$ takes the following form: 
\begin{align}
\mbox{$\rvA x =\sum_{i=1}^n x_i v_i$,\qquad  where $\{v_i\}_{i=1}^n \subseteq \Gamma\setminus\{0\}$ and $\sum_{i=1}^n  v_i\in \calC(d)$ (= $\inter \Gamma$).} \label{eq:def_case_i_A}
\end{align}


\begin{remark}
Note that the objectives in both~\eqref{eq:PET} and~\eqref{eq:DOPT} fall under this function  class. Specifically, for~\eqref{eq:PET}, $\bbY = \bbR^m$, $p(y) = \prod_{j=1}^m y_j$ with $d = (1,1,\ldots,1)$ and $\calC(d) = \bbR^m_{++}$, 
and for $i\in[n]$,  $v_i$  is the $i$-th column of the (nonnegative) data matrix $A = [a_1 \cdots a_m]^\top$. For~\eqref{eq:DOPT}, $\bbY = \bbS^m$, $p(Y) = \det(Y)$ with $d = I_m$ and $\calC(d) = \bbS^m_{++}$, 
and for $i\in[n]$,  $v_i=A_i$. 
\end{remark}




\noindent 

\begin{remark} \label{rmk:Gamma_regular}
Given a nonempty convex set $ \calC\subseteq \bbY$, let $\calL(\calC)$ denote the lineality space of $\calC$, i.e., 
$\calL(\calC):= \{h\in\bbY:y+th\in \calC, \,\forall\,t\in \bbR, \,\forall\,y\in\calC \}$. 
From~\cite[Fact~2.9]{Bauschke_01a}, we know that  
\begin{equation}
\calZ_p(d) = \calL(\calC(d)) = \calL(\Gamma), \label{eq:Z_p} 
\end{equation}
Therefore,  $\Gamma$ is pointed if and only if $p$ is complete. Since $\calC(d)$ is open and convex, we know that $\Gamma $ is closed, convex and solid, and hence $\Gamma$ is regular if and only if $p$ is complete.
\end{remark}

\begin{remark}\label{rmk:LHSCB}
From~\cite[Section~4]{Guler_97} and~\cite[Chapter~2]{Nest_94}, we know that $\barf$ 
is an 
{\em $m$-logarithmically-homogeneous self-concordant barrier} ($m$-LHSCB) on $\Gamma$. 
Specifically, this means  that $\barf$ is convex, three-times differentiable on $\calC(d)$,  and satisfies the following conditions:  
\begin{enumerate}[label=\roman*),leftmargin=20pt,itemsep=5pt,parsep=0pt,topsep=5pt]
\item for any $\{y_k\}_{k\ge 0}\subseteq \calC(d)$ such that $y_k\to y\in \bdry \calC(d)$, we have $\barf(y_k)\to +\infty$,
\item $\abs{D^3 \barf(y)[h,h,h]} \le 2\, \big(D^2 \barf(y)[h,h]\big)^{3/2}$, for all $y\in \calC(d)$ and $h\in\bbY$,
\item $\barf(ty) = \barf(y) - m\ln t$, for all $y\in \calC(d)$ and $t>0$,
\end{enumerate}
where for $k\ge 1$ and $h_1,\ldots,h_k\in\bbY$,   
\begin{equation}
D^k \barf(y)[h_1,\ldots,h_k]:= \textstyle \frac{\partial^k}{\partial s_1\cdots  \partial s_k} \; \barf(y + s_1h_1 + \cdots + s_kh_k)\big|_{s_1=\cdots=s_k=0}\label{eq:direc_deriv}
\end{equation}
denotes the $k$-th order directional derivative of $\barf$ at $y\in \calC(d)$ in the directions $h_1,\ldots,h_k\in\bbY$. 
From the these definitions, it is clear that $\barf$ is essentially smooth (cf.~Definition~\ref{def:legendre}). 
\end{remark}


\subsubsection{Proving That $f$ and $\rvA$ Satisfy Assumptions~\ref{assump:A} and~\ref{assump:F} }

Let us start with the  following proposition regarding $\barf$, whose proof can be found in Appendix~\ref{app:proof_func}. 


\begin{prop} \label{prop:equiv_barf}
For $\barf$ defined in~\eqref{eq:def_case_i}, the following four statements are equivalent:\\
\begin{tabular}{ll}
(a) $p$ is complete,&(b) $\barf$ is non-degenerate (i.e., $\nabla^2 \barf(y)\succ 0$,  $\forall\,y\in \calC(d)$),\\ 
(c) $\barf$ is essentially strictly convex, & (d) $\barf$ is Legendre.
\end{tabular}
 \end{prop}


\noindent 
In particular, Proposition~\ref{prop:equiv_barf} implies that the Legendre assumption of $\barf $ amounts to assuming  $p$ to be complete. In this case, we know that 
$\nabla f: \calC(d) \to \inter \Gamma^*$, 
and hence 
\begin{equation}
\ipt{\nabla f(y)}{a}>0, \quad \forall\, y\in \calC(d), \quad \forall\, a\in \Gamma\setminus\{0\}. 
\label{eq:inner_prod_pos}
\end{equation}

\noindent 
To show that $f$ and $\rvA$ satisfy Assumptions~\ref{assump:A} and~\ref{assump:F}, we need the following two lemmas. 

\begin{lemma} \label{lem:Q_a}
Let $p$ be complete. 
For any $a\in \Gamma\setminus\{0\}$, define 
\begin{equation}
Q_a(y):= \ln(\ip{\nabla f(y)}{a}), \quad \forall \,y\in \calC(d).  \label{eq:def_Q_a}
\end{equation}
Then  $Q_a$ is convex on $\calC(d)$, 
for any $a\in \Gamma\setminus\{0\}$. 
\end{lemma}

\begin{proof}\renewcommand{\qedsymbol}{}

 For $k\in[m]$, let $E_{k}$ be the elementary symmetric polynomial of degree $k$ on $\bbR^m$, i.e., 
\begin{equation}
E_{k}(\lambda):= \textstyle \sum_{\calI\subseteq [m],\, \abst{\calI} = k}\;  \prod_{i\in\calI} \; \lambda_i, \quad \forall\, \lambda\in \bbR^m. 
\end{equation}
Fix any $a\in \calC(d)$.  From~\cite[Fact~2.10]{Bauschke_01a}, we know that $E_{m-k}(\lambda^a(y)) = D^k p(y)[a]^k/(k!\, p(a))$ for any $y\in \bbY$ and $k=0,1,\ldots,m-1$, and consequently, 
\begin{equation}
Q_a(y)= \ln\frac{\ipt{\nabla p(y)}{a}}{p(y)}= \ln\frac{E_{m-1}(\lambda^a(y))}{E_m(\lambda^a(y))}, \quad \forall\,y\in \calC(d). 
\end{equation} 
In addition, from~\cite[Corollary~4.6]{Bauschke_01a}, we know that for any $a\in \calC(d)$, $Q_a$ is convex on $\calC(d)$. 
Now, for any $a\in \Gamma\setminus\{0\}$, there exist $\{a_k\}_{k\ge 0}\subseteq \calC(d)$ such $a_k\to a$, 
and since for any $y\in\calC(d)$, $a\mapsto Q_a(y)$ is continuous on $\Gamma\setminus\{0\}$ (cf.~\eqref{eq:inner_prod_pos}), 
we know that $Q_{a_k}(y) \to Q_a(y)$. Therefore, for any $y,y'\in\calC(d)$ and $\alpha\in[0,1]$, we have 
\begin{align*}
Q_a(\alpha y +(1-\alpha)y') &= \lim_{k\to +\infty} Q_{a_k}(\alpha y +(1-\alpha)y')\\
 &\le \lim_{k\to +\infty} \alpha Q_{a_k}( y) +(1-\alpha)Q_{a_k}( y') = \alpha Q_a( y) +(1-\alpha)Q_a( y').
 \tag*{$\square$}
\end{align*}
\end{proof}
\vspace{-4ex}
\noindent 

\begin{lemma} \label{lem:sum_v_i}
Let $\Omega\subseteq \bbY$ be a closed, convex and solid cone, 
and $\{v_i\}_{i=1}^n\subseteq \Omega$. If $\sum_{i=1}^n v_i \in\inter\Omega$, then $\sum_{i=1}^n \alpha_i v_i\in\inter \Omega$ for any $\alpha_i>0$, $i\in[n]$.  
\end{lemma}
\begin{proof}
See Appendix~\ref{app:proof_func}. 
\end{proof}

\noindent 
Equipped with the two lemmas above, we can show the following. 

\begin{prop}
Let $p$ be complete, and $f$ and $\rvA$ be given by~\eqref{eq:def_case_i} and~\eqref{eq:def_case_i_A}, respectively. Then $f$ and $\rvA$ satisfy
Assumption~\ref{assump:A} and $F = f\circ \rvA$ satisfies Assumption~\ref{assump:F}. 
\end{prop}

\begin{proof}
From Remark~\ref{rmk:Gamma_regular}, since $p$ is complete, $\Gamma$ is regular. 
Since $\{v_i\}_{i=1}^n \subseteq \Gamma$ and $\sum_{i=1}^n  v_i\in \inter \Gamma$, 
by Lemma~\ref{lem:sum_v_i}, we know that $\rvA(\bbR_{++}^n)\subseteq \inter\Gamma$. 
In addition, since the adjoint $\rvA^*: \bbY^*\to \bbR^n$ has the form $\rvA^* z = (\ipt{z}{v_i})_{i=1}^n$ and $\{v_i\}_{i=1}^n \subseteq \Gamma\setminus\{0\}$, we know that $\rvA^*(\inter\Gamma^*) \subseteq \bbR_{++}^n$. This shows that $f$ and $\rvA$ satisfy Assumptions~\ref{assump:A}. Now, consider $\phi_u$ defined in~\eqref{eq:phi_u}, and for any $i\in[n]$, 
\begin{align*}
\phi_i(x):=\phi_{e_i}(x)&= \ipt{\ln(\nabla F(x))}{e_i} = \ln(\nabla_i F(x)) = \ln(\ipt{\nabla F(x)}{e_i}) \\
&\qquad = \ln(\ipt{\nabla f(\rvA x)}{\rvA e_i}) = \ln(\ipt{\nabla f(\rvA x)}{v_i}) = Q_{v_i}(\rvA x), \quad 
\end{align*}
where $Q_{v_i}$ is defined in~\eqref{eq:def_Q_a}. 
Since $\rvA(\bbR_{++}^n)\subseteq \inter\Gamma = \calC(d)$ and $\{v_i\}_{i=1}^n \subseteq \Gamma\setminus\{0\}$, invoking Lemma~\ref{lem:Q_a}, we know that $\phi_i$ is convex on $\bbR_{++}^n$, for any $i\in[n]$. Since $\PI(\bbR^n)=  \{e_i\}_{i=1}^n$, by Proposition~\ref{prop:LGrad_convex}, we know that $F$ satisfies Assumption~\ref{assump:F}. 
%
%
%
%
%
%
%
\end{proof}

\noindent 
To introduce the second scenario, we need to first introduce the notion of  {\em representative simple Euclidean Jordan Algebra} and a Cauchy–Schwarz inequality in this algebra. 

\subsection{A Cauchy–Schwarz Inequality in  Representative Simple Euclidean Jordan Algebras} \label{sec:CS_EAJ} 

Our description of the representative simple Euclidean Jordan algebras follows from~\cite[Section~2]{Schm_01}.  
Let $\bbA$ be a finite-dimensional real vector space. 
An algebra $(\bbA,\cdot)$ is an associative $\bbR$-algebra 
if the bilinear product $\cdot:\bbA\times\bbA\to\bbA$ associative, 
i.e., $(x\cdot y)\cdot z = x\cdot (y\cdot z)$ for all $x,y,z\in\bbA$.  
For convenience, we will omit $\cdot$ and simply write $x\cdot y$ as $xy$ for $x,y\in\bbA$. 
Denote the identity element in $\bbA$ by $e$, and we assume that 
there exists a {\em linear automorphism} $(\cdot)'$ on $\bbA$, called {\em conjugation}, 
such that  $(x')' = x$ and $(xy)'= y'x'$ for any $x,y\in\bbA$. We call $x'\in\bbA$ the {\em adjoint} of $x$ and $x\in\bbA$ {\em self-adjoint} if $x'=x$. We denote the subspace of self-adjoint elements in $\bbA$ by $\bbS$, namely $\bbS:=\{x\in\bbA:x'=x\}$. Note that $e = e'$ (since $e'x = xe' = x$ for all $x\in\bbA$), and hence $e\in \bbS$. 
In addition, we call $x\in\bbA$ invertible if there exists $y\in\bbA$ such that $xy=yx=e$. Note that since $(\bbA,\cdot)$ is associative, if such $y$ exists, it must be unique, and we call $y$ the inverse of $x$ and denote it by $x^{-1}$. 

Based on the associative product $\cdot$, we can define the Jordan product 
\begin{equation}
x\circ y:= (xy+yx)/2, \quad \forall\, x,y\in\bbA, 
\end{equation}
 and call  $(\bbA,\circ)$ the Jordan algebra derived from $(\bbA,\cdot)$. Note that $(\bbA,\circ)$ has the same identity as $(\bbA,\cdot)$, namely $e$, and $P(y) x= yxy$ for all $x,y\in\bbA$, where $P(\cdot)$ is defined in~\eqref{eq:def_P}.  
 Also,  since 
$\bbS$ is closed under $\circ$ 
 and contains $e$, 
 $(\bbS,\circ)$ is a Jordan sub-algebra of $(\bbA,\circ)$ with identity $e$. Based on the definitions above, let us define the {\em simple Euclidean Jordan associative} (EAJ) system. 

\begin{definition}[Simple EAJ system] \label{def:EAJ_system}
A triple $(\bbA,\bbS,\bbV)$ is called a {simple} EAJ system if 
\begin{enumerate}[label=\roman*)]
\item $(\bbA,\cdot)$ is an associative $\bbR$-algebra with identity element $e$ and conjugation $(\cdot)'$ 
such that $\tr(xx')\ge 0$ for all $x\in\bbA$ (where $\tr(\cdot)$ is defined in~\eqref{eq:def_det}), 
\item the Jordan algebra $(\bbS,\circ)$ is Euclidean, and hence a Euclidean Jordan algebra, 
\item $\bbV$ is a linear subspace of $\bbS$ that contains $e$ and is closed under $\circ$, thus $(\bbV,\circ)$ is a Jordan algebra. In addition, $(\bbV,\circ)$ is Euclidean and simple, making it a {\em simple} Euclidean Jordan sub-algebra of $(\bbS,\circ)$,
\item $\bbA$ is generated by $\bbV$, namely, the elements of $\bbA$ consist of all the linear combinations of the {\em associative} products of elements in $\bbV$.  
\end{enumerate}
\end{definition}

Following
~\cite[Section~4]{Schm_01}, we call 
$(\bbV,\circ)$ a {\em representable} simple Euclidean Jordan algebra if there exists a {simple} EAJ system $(\bbA,\bbS,\bbV)$ as described above. Note that among all the five simple EJAs listed in Section~\ref{sec:background}, only the first four are representable, namely $\bbL^{n+1}$ and $\mathrm{Herm}(n,\bbF)$, where $\bbF = \bbR$, $\bbC$ or $\bbH$. 
In addition, we call the cone of squares in any representable simple  Euclidean Jordan algebra as a  {\em representable irreducible symmetric cone}. Notation-wise, denote the cone of squares in $(\bbS,\circ)$ and $(\bbV,\circ)$ by $\calK(\bbS)$ and $\calK(\bbV)$, respectively.

The main result in this section is the following Cauchy–Schwarz inequality over a {simple} EAJ system, and its corollary that concerns a representable simple Euclidean Jordan algebra. The corollary will be useful in Section~\ref{sec:second_case}, where we introduce the second scenario. In addition, both results may be of independent interest, and their proofs are shown in Section~\ref{sec:proof_CS}.  

\begin{theorem}
\label{thm:CS0}
Let $(\bbA,\bbS,\bbV)$ be a {simple} EAJ system,  
and $a_i,b_i\in \bbA$ for $i\in[m]$ such that $\sum_{i=1}^m a_ia_i' \in \inter\calK(\bbS)$. Then we have 
\begin{equation}
\textstyle \sum_{i=1}^m b_ib_i'\; \succeq_{\calK(\bbS)}\; (\sum_{i=1}^m b_i a_i') (\sum_{i=1}^m a_ia_i')^{-1}(\sum_{i=1}^m  a_ib_i'). 
\label{eq:CS0}
\end{equation}
\end{theorem}

\begin{corollary}
\label{cor:CS}
Let $(\bbV,\circ)$ be a representable simple Euclidean Jordan algebra. 
For $i\in[m]$, let $v_i\in\calK(\bbV)$ 
and  $\alpha_i,\beta_i\in\bbR$  such that $\sum_{i=1}^m \alpha_i^2 v_i\in\inter\calK(\bbV)$.  
Then we have 
\begin{equation}
\hspace{-2ex} \textstyle \sum_{i=1}^m \beta_i^2 v_i\; \succeq_{\calK(\bbV)}
P\big(\sum_{i=1}^m {\alpha_i\beta_i} v_i\big) (\sum_{i=1}^m \alpha_i^2 v_i)^{-1}. \label{eq:CS}
\end{equation}
\end{corollary}

\begin{remark}
In Corollary~\ref{cor:CS},  if $(\bbV,\circ)=\mathrm{Herm}(n,\bbR)$ or $ \mathrm{Herm}(n,\bbC)$, 
then~\eqref{eq:CS} reduces to the so-called ``matrix Cauchy-Schwarz inequality'', which can be proved using simple matrix algebra (see e.g.,~\cite{Lav_08,Zhao_25}).  However, note that for the Jordan spin algebra $\bbL^{n+1}$, it is rather challenging to prove~\eqref{eq:CS} directly, and leveraging  the simple EJA system seems a much viable approach. In addition, this approach unifies the proof of~\eqref{eq:CS} for all the representable simple Euclidean Jordan algebras. 
\end{remark}


\subsubsection{Proofs of Theorem~\ref{thm:CS0} and Corollary~\ref{cor:CS} } \label{sec:proof_CS}

Our proofs make use of a few technical lemmas. Before showing 
them, 
let us make some remarks regarding the simple EAJ system $(\bbA,\bbS,\bbV)$. 

First, for any $x\in\bbA$,  since $x\circ x = x  x$, 
both products $\circ$ and $\cdot$ define the same power $x^q$ for any 
$q\in\bbZ_+$  (recall that $x^0:=e$), 
and hence the expression
 $x^q$ has no ambiguity. 
This fact has the following consequences. First, both $\circ$ and $\cdot$  define the same characteristic polynomial of $x$, and hence the eigenvalues of $x$, as well as  $\tr(x)$ and $\det(x)$, are defined in the same way w.r.t.\ these two products. Second, 
both $\circ$ and $\cdot$ define the same set of real polynomials in $x$, 
namely $\bbR[x]$. 
Since $p_1(x) p_2(x) = p_1(x)\circ p_2(x)$ for any $p_1(x), p_2(x)\in\bbR[x]$,  
we know that $(\bbR[x],\cdot)$ and $(\bbR[x],\circ)$ are the same $\bbR$-algebra 
(which is commutative and associative). 

Second, suppose that $x$ is invertible in $(\bbV,\circ)$ with inverse $x^{-1}$. By definition, 
 $x^{-1}\in \bbR[x]$ 
 and $x^{-1}\circ x = e$. Since $\bbR[x]\subseteq \bbV\subseteq \bbS\subseteq \bbA$, we know that $x^{-1}$ is also the inverse of $x$ in $(\bbS,\circ)$ and $(\bbA,\circ)$. 
 Moreover, since $(\bbR[x],\cdot)$ is commutative, we have 
$e = x\circ x^{-1} = ({xx^{-1} + x^{-1}x})/{2} = xx^{-1} = x^{-1}x.$ 
Therefore, $x^{-1}$ is also the inverse of $x$ in $(\bbA,\cdot)$. This means that for $x\in \bbV$, the expression $x^{-1}$ has no ambiguity. The same argument also reveals that  for invertible  $x$ in  $(\bbS,\circ)$, its inverse $x^{-1}$ in  $(\bbS,\circ)$ also coincides with that in $(\bbA,\circ)$ and $(\bbA,\cdot)$, and has no ambiguity.  
\\[-1.5ex]

\noindent 
Our technical lemmas start with the following  results from~\cite[Section~2]{Schm_01}.

\begin{lemma}[{\cite[Section~2]{Schm_01}}] \label{lem:Alizadeh}
Let $(\bbA,\bbS,\bbV)$ be a {simple} EAJ system. Then 
 \begin{enumerate}[label=(\roman*)]
\item \label{item:psd_in_S} for any $x\in\bbA$, $xx'\in\calK(\bbS)$, 
\item \label{item:spec_decomp} for any $x\in\bbV$, let its spectral decomposition in $(\bbV,\circ)$ be  $x=\sum_{i=1}^n \mu_i q_i$, then its spectral decomposition in $(\bbS,\circ)$ is $x=\sum_{i=1}^n \mu_i (\sum_{j=1}^k d_{ij})$ for some positive integer $k$. 
\end{enumerate}
\end{lemma} 

\noindent 
In particular, Lemma~\ref{lem:Alizadeh}\ref{item:spec_decomp} implies that for any $x\in\bbV$, if $\mu$ is an eigenvalue of $x$ in  $(\bbV,\circ)$ with multiplicity $m$, then it is an eigenvalue of $x$ in  $(\bbS,\circ)$ with multiplicity $km$. Based on Lemma~\ref{lem:Alizadeh}, 
we can prove 
the following lemma. 

\begin{lemma}\label{lem:cone_intersect}
Consider the setting in Lemma~\ref{lem:Alizadeh}. 
Then 
$\calK(\bbV) = \bbV\cap \calK(\bbS)$ and $\ri\calK(\bbV) = \bbV\cap {\sf int}\,\calK(\bbS)$, where ${\sf int}$ is taken w.r.t.\ $\bbS$. 
\end{lemma}
\begin{proof}
Given any $x\in\bbV$, note that $x\in \calK(\bbV)$ (resp.\ $\calK(\bbS)$) if and only if the eigenvalues of $x$ in $\bbV$ (resp.\ $\bbS$) are nonnegative. Similarly, $x\in \ri\calK(\bbV)$ (resp.\ $\inter\calK(\bbS)$) if and only if the eigenvalues of $x$ in $\bbV$ (resp.\ $\bbS$) are positive. 
Using Lemma~\ref{lem:Alizadeh}\ref{item:spec_decomp}, we complete the proof.
\end{proof}

\noindent 

\begin{lemma} \label{lem:sum_v_i=0}
Let $\calK\ne \emptyset$ be a closed, convex and pointed cone, and $\{v_i\}_{i=1}^n\subseteq \calK$. If $\sum_{i=1}^n v_i = 0$, then $v_i=0$ for all $i\in[n]$.
\end{lemma}

\begin{proof}
Fix any $j\in[n]$. Since $\sum_{i=1}^n v_i = 0$, we have $ v_j = -\sum_{i\ne j}^n v_i$. Since $\{v_i\}_{i=1}^n\subseteq \calK$ and $\calK$ is convex, we have $v_j\in \calK\cap (-\calK)$. Since $\calK$ is pointed, we have $v_j=0$. 
\end{proof}



\noindent 
{\bf Proof of Theorem~\ref{thm:CS0}.}
From Lemma~\ref{lem:Alizadeh}\ref{item:psd_in_S},  we have
\begin{align}
 \textstyle \sum_{i=1}^m (b_i-c a_i)(b_i-c a_i)'\;\succeq_{\calK(\bbS)}\;0. \label{eq:sum_rank_one}
\end{align}
On the other hand, note that
\begin{align*}
 \textstyle\sum_{i=1}^m (b_i-ca_i) (b_i-ca_i)'=\; &\textstyle\sum_{i=1}^m (b_i-ca_i) (b_i'- a_i'c')\\
= \; & \textstyle\sum_{i=1}^m b_i b_i' - (\sum_{i=1}^m b_i a_i') c' - c  (\sum_{i=1}^m a_ib_i') + c (\sum_{i=1}^m a_ia_i')c'.
\end{align*}
Since $\sum_{i=1}^m a_ia_i' \in \inter\calK(\bbS)$, it is invertible in $(\bbS,\circ)$. 
Therefore, let $c=\textstyle (\sum_{i=1}^m b_ia_i')(\sum_{i=1}^m a_ia_i')^{-1},$  
 and 
 we have 
\begin{align}
 \textstyle\sum_{i=1}^m (b_i-ca_i) (b_i-ca_i)'
= \;  \textstyle\sum_{i=1}^m b_i b_i'  - (\sum_{i=1}^m b_ia_i')(\sum_{i=1}^m a_ia_i')^{-1} (\sum_{i=1}^m a_ib_i'). \label{eq:sum_rank_one2}
\end{align}
Combining~\eqref{eq:sum_rank_one} and~\eqref{eq:sum_rank_one2}, we complete the proof. \qed

\vspace{1ex}

\noindent 
{\bf Proof of Corollary~\ref{cor:CS}.}
Let $(\bbA,\bbS,\bbV)$ be the {simple} EAJ system associated with $(\bbV,\circ)$. For $i\in[m]$, since  $v_i\in\calK(\bbV)$, 
define $a_i: = {\alpha_i}v_i^{1/2}$ and $b_i: = {\beta_i}v_i^{1/2}$.   Thus $\sum_{i=1}^m a_ia_i' = \sum_{i=1}^m \alpha_i^2 v_i\in{\sf int}\,\calK(\bbV)$, where ${\sf int}$ is taken w.r.t.\ $\bbV$.  
By Lemma~\ref{lem:cone_intersect}, we have  $\sum_{i=1}^m a_ia_i'\in\inter\calK(\bbS)$, where ${\sf int}$ is taken w.r.t.\ $\bbS$. 
Invoking Theorem~\ref{thm:CS0}, 
we have 
\begin{align}
\begin{split}
  &\textstyle\sum_{i=1}^m \beta_i^2 v_i\; \succeq_{\calK(\bbS)}(\textstyle\sum_{i=1}^m b_i a_i') (\sum_{i=1}^m a_ia_i')^{-1}(\sum_{i=1}^m  a_ib_i')\\
&\qquad = \textstyle(\sum_{i=1}^m \alpha_i \beta_i v_i) (\sum_{i=1}^m \alpha_i^2 v_i)^{-1}(\sum_{i=1}^m  \alpha_i \beta_i v_i)
=P\big(\sum_{i=1}^m {\alpha_i\beta_i} v_i\big) (\sum_{i=1}^m \alpha_i^2 v_i)^{-1}.
\end{split}
\label{eq:CS.5}
\end{align}
%
Note that both sides of~\eqref{eq:CS.5} lie in $\bbV$, hence  by Lemma~\ref{lem:cone_intersect},~\eqref{eq:CS.5} is equivalent to~\eqref{eq:CS}. \qed\\[-1ex] 

\noindent 
Finally, we end this section with the following lemma, which will be  useful in Section~\ref{sec:proof_barf}. 

\begin{lemma} \label{lem:P(y)}
Consider the setting in Lemma~\ref{lem:Alizadeh}, and let $\succeq$ be induced by ${\calK(\bbV)}$. 
 If $x\succeq z$ for some $x,z\in\bbV$, then  $P(y)x \succeq P(y)z$ for any $y\in\bbV$. 
\end{lemma}
\begin{proof}
It suffices to show that if $x\succeq 0$ for some $x\in\bbV$, 
then $P(y)x \succeq 0$ for any $y\in\bbV$. 
Recall that $P(y)x = yxy$ for all $x,y$ in $(\bbA,\circ)$. 
If $x\in \calK(\bbV)$, then $x=u^2$ for some $u\in\bbV$. 
Since both $y,u\in \bbV\subseteq\bbS$, by Lemma~\ref{lem:Alizadeh}\ref{item:psd_in_S}, we have $P(y)x = yuuy= yu(yu)'\in\calK(\bbS)$. Since $P(y)x\in \bbV$, from Lemma~\ref{lem:cone_intersect}, we know that $P(y)x\in \calK(\bbV)$. 
\end{proof}


\subsection{Second Scenario} \label{sec:second_case}

Let $(\bbV,\circ)$ be a representable simple Euclidean Jordan algebra (cf.\ Section~\ref{sec:CS_EAJ}) and $\calK:= \calK(\bbV)$, i.e.,  the cone of squares in $\bbV$. 

\subsubsection{Definitions of $f$ and $\rvA$}  \label{sec:def_f_A_case_ii}
let $\bbY:= \bbR^m$ be equipped with the standard inner product $\ipt{\cdot}{\cdot}$, i.e., $\ipt{y}{u} = \sum_{i=1}^m y_i u_i$ for $y,u\in \bbR^m$. Given a three-times differentiable function $\psi: \bbR^m_{++} \to \bbR$, 
 we say that $\psi$ satisfies the ``coordinate-wise self-dominance'' condition on $\bbR^m_{++}$, if the following holds: 
 \begin{equation}
D \psi(y)[e_j]\; D^3 \psi(y)[e_j,h,h]\ge 2(D^2 \psi(y)[e_j,h])^2, \quad  \forall\,y\in\bbR^m_{++}, \;\;  \forall\, h\in  \bbR^m, \;\;\forall\, j\in[m].   \label{eq:cond_psi}
\end{equation}
The class of functions $f$ that we 
focus on have the following properties:  
\begin{enumerate}[label=(P\arabic*),leftmargin = 26pt]
\item \label{item:legendre_LH_-f} $ \barf:\bbY\to \barbbR$  is Legendre and $\theta$-LH with $\inter\dom \barf = \bbR_{++}^m$ (where $\barf:=-f$), 
\item \label{item:f_3times_diff} $f$  is three-times differentiable on $\bbR^m_{++}$,
\item  \label{item:condition_psi}  either $f$ satisfies~\eqref{eq:cond_psi}, 
or the function   $\exp\circ (\rho f):\bbR^m_{++}\to \bbR_{++}$ is concave on $\bbR^m_{++}$  and satisfies
 ~\eqref{eq:cond_psi}, for some $\rho>0$.   
\end{enumerate}
As a result of~\ref{item:legendre_LH_-f}, $\Gamma: =\cl\dom f = \bbR_+^m$, $\Gamma^* = \bbR_+^m$, and $\nabla f: \bbR_{++}^m\to \bbR_{++}^m$ (cf.~Proposition~\ref{prop:well_posed}). 
We let the linear operator 
$\rvA:\bbV\to \bbR^m$ take the following form:
\begin{align}
\mbox{$\rvA x := (\ipt{x}{v_i})_{i=1}^m$ for $x\in \bbV$,\qquad  where $\{v_i\}_{i=1}^m\subseteq \calK\setminus\{0\}$ and $\sum_{i=1}^m v_i \in \inter\calK$.} \label{eq:def_case_ii_A}
\end{align}
Consequently, the adjoint $\rvA^*: \bbR^m\to \bbV$ has the form $\rvA^* y =\sum_{i=1}^m y_i v_i$.\\[-1ex]

\noindent 
{\bf Examples of $f$.}  
We  provide two examples of $f$ that satisfy~\ref{item:legendre_LH_-f} to~\ref{item:condition_psi}. 
The first example is 
\begin{equation}
f_1(y):= \textstyle\sum_{j=1}^m p_j \ln y_j, \quad \forall\,y\in\bbR_{++}^m,\quad \mbox{where } p_j >0, \;\forall\,j\in[m]. \label{eq:f_sum_logs}
\end{equation}
The separable structure of $f_1$ makes it easy to verify that it 
satisfies~\eqref{eq:cond_psi}.  
Note that this example precisely appears in~\eqref{eq:QST}, where $(\bbV,\circ) = \mathrm{Herm}(n,\bbC)$, $\calK = \bbH_+^n$, and the linear operator $\rvA$ has the form in~\eqref{eq:def_case_ii_A} with $v_i = A_i$ for $i\in[m]$. 
 Now, for 
 $p\in(0,1)$, let $\normt{y}_p := (\textstyle\sum_{j=1}^m y_j^p)^{1/p}$ be the $p$-quasi-norm  of  $y\in \bbR_+^m$.  
The second example of $f$ 
is 
\begin{equation}
f_2(y): = \ln\normt{y}_p := (1/p)\ln(\textstyle\sum_{j=1}^m y_j^p), 
\quad \forall\,y\in\bbR_{+}^m\setminus\{0\}, \quad\where\; p\in(0,1].  \label{eq:log_p_norm}
\end{equation}  
Indeed, since $\exp\circ (p f_2)$  has a separable structure, one can  easily check that it is concave and 
satisfies~\eqref{eq:cond_psi}.  
Note that this example appears in~\eqref{eq:DBQP} with $m=n$ and $p=1/2$, wherein $(\bbV,\circ) = \mathrm{Herm}(n,\bbR)$, $\calK = \bbS_+^n$, and 
 $\rvA$ has the form in~\eqref{eq:def_case_ii_A} with $v_i = q_iq_i^\top $ for $i\in[n]$. 

\subsubsection{Proving $f$ and $\rvA$ Satisfy Assumptions~\ref{assump:A} and~\ref{assump:F} }

We start by showing that $f$ and $\rvA$ satisfy Assumption~\ref{assump:A}.

\begin{prop} \label{prop:assumption_A}
Let $f$ and $\rvA$ be given in Section~\ref{sec:def_f_A_case_ii}. Then they satisfy Assumption~\ref{assump:A}. 
\end{prop}

\begin{proof}
Since $\{v_i\}_{i=1}^m \subseteq \calK\setminus\{0\}$ and $\calK$ is self-dual, it is clear that $\rvA(\inter\calK) \subseteq \bbR_{++}^m$. 
  Since $\calK$ is regular, $\{v_i\}_{i=1}^n \subseteq \calK$ and $\sum_{i=1}^n  v_i\in \inter \calK$, by Lemma~\ref{lem:sum_v_i}, 
  we know that $\rvA^* y =\sum_{i=1}^m y_i v_i\in \inter\calK$ for any $y\in \bbR_{++}^m$. Hence $\rvA^*(\inter\Gamma^*) = \rvA^*(\bbR_{++}^m)\subseteq \inter\calK$. 
\end{proof}



\noindent 
To show that $f$ and $\rvA$ satisfy Assumption~\ref{assump:F}, we need the following important result. The proof of this result is relatively long and will be provided in Section~\ref{sec:proof_barf}. 

\begin{theorem} \label{thm:barf}
%
%
Let $\psi:\bbR_{++}^m\to\bbR$ be 
three-times differentiable on $\bbR^m_{++}$ 
and satisfies~\eqref{eq:cond_psi}.  Also, assume that $\nabla \psi: \bbR^m_{++}\to \bbR^m_{++}$. 
Define 
 $g:\bbR^m_{++}\to \bbV$ as 
\begin{equation}
g(y):= \ln\big(\rvA^*\nabla \psi(y)\big) = \textstyle \ln\big(\sum_{j=1}^m \nabla_j \psi(y)v_j\big), \quad \forall\,y\in\bbR^m_{++}.  \label{eq:def_g}
\end{equation}
Then $g$ is twice differentiable on $\bbR^m_{++}$, and for any $y\in\bbR^m_{++}$ and $h\in\bbR^m$, $D^2g(y)[h,h]\in\calK$, where $D^2g(y)[h,h]$ is defined in a similar way as in~\eqref{eq:direc_deriv}. 
\end{theorem}

\noindent 
Based on Theorem~\ref{thm:barf}, we can easily show that $f$ and $\rvA$ satisfy Assumption~\ref{assump:F}. 

\begin{corollary}
Let $f$ and $\rvA$ be given in Section~\ref{sec:def_f_A_case_ii}. 
Then $F:= f \circ \rvA$ 
satisfies Assumption~\ref{assump:F}. 
\end{corollary}

\begin{proof}
Fix any $x\in\inter\calK$ and $d\in\bbV$. 
Since $\nabla F: \inter\calK\to \inter\calK$ (cf.\ Proposition~\ref{prop:well_posed}), we define $\eta_x(\alpha):= \ln(\nabla F(x+\alpha d)) $ with $\dom \eta_x:= \{\alpha\in\bbR: x+\alpha d\in\inter\calK\}$, and also 
$\barg(y):= \ln(\rvA^*\nabla f(y))$ for $y\in\bbR^m_{++}$. 
Note that  $\eta_x(\alpha)= \ln(\rvA^*\nabla f(\rvA x + \alpha \rvA d) ) = \barg(\rvA x + \alpha \rvA d)$, and hence 
\begin{align}
D^2\ln(\nabla F(x))[d,d] = \eta_x''(0) = D^2\barg(\rvA x)[\rvA d,\rvA d]. \label{eq:D^2ln_barg}
\end{align}
Note that by definition, $f$ satisfies~\ref{item:legendre_LH_-f} to~\ref{item:condition_psi}. 
Now, define $\psi := f$ if $f$ satisfies~\eqref{eq:cond_psi}, and $\psi := \exp\circ (\rho f)$ if $\exp\circ (\rho f)$ is concave on $\bbR_{++}^m$ and satisfies~\eqref{eq:cond_psi}, for some $\rho>0$. Note that in either case, $\psi$ is concave  on $\bbR_{++}^m$, and satisfies all the assumptions in Theorem~\ref{thm:barf}. Let $g$ be defined in~\eqref{eq:def_g}, and we shall show that
\begin{equation}
D^2 \barg(y)[h,h]\succeq D^2 g(y)[h,h] , \quad \forall\,  y\in   \bbR_{++}^m, \; \; \forall\, h\in \bbR^m,  \label{eq:barg_g}
\end{equation}
 where $\succeq$ is induced by $\calK$. Since~\eqref{eq:barg_g} is trivial when $\psi=f$, 
 we only focus on 
 $\psi = \exp\circ (\rho f)$. In this case, $f(y) = \rho^{-1} \ln\psi(y)$ for $y\in   \bbR_{++}^m$, and hence  
 \begin{equation}
\barg(y) =  \ln\big(\rvA^*\nabla \psi(y)/(\rho\psi(y))\big) = \ln\big(\rvA^*\nabla \psi(y)\big) - \ln\big(\rho\psi(y)\big)\;  e = g(y) - \kappa(y)\; e,  
\end{equation}
where $\kappa:= \ln\circ(\rho\psi): \bbR_{++}^m\to \bbR$ is concave on $\bbR^m_{++}$, since $\psi$ is concave on $\bbR^m_{++}$. As a result, 
\begin{equation}
D^2 \barg(y)[h,h] =  D^2 g(y)[h,h] -  D^2 \kappa(y)[h,h] e\succeq D^2 g(y)[h,h], \quad \forall\,  y\in   \bbR_{++}^m, \; \; \forall\, h\in \bbR^m.  
\end{equation}
Now, consider $\phi_u$ defined in~\eqref{eq:phi_u}. By~\eqref{eq:D^2ln_barg}, for any $x\in\inter\calK$, $d\in\bbV$ and $u\in\PI(\bbV)\subseteq \calK$, 
\begin{align}
D^2 \phi_u(x)[d,d] &= \ipt{D^2\ln(\nabla F(x))[d,d]}{u} = \ipt{D^2\barg(\rvA x)[\rvA d,\rvA d]}{u}\nn \ge \ipt{D^2 g(\rvA x)[\rvA d,\rvA d]}{u}\ge 0, 
\end{align} 
where the first and second inequalities follow from~\eqref{eq:barg_g} and Theorem~\ref{thm:barf}, respectively. 
\end{proof}

\subsubsection{Proof of Theorem~\ref{thm:barf} } \label{sec:proof_barf}

To prove Theorem~\ref{thm:barf}, we need the following important properties of the quadratic representation 
$P(x)$ of $x\in\bbV$. 

\begin{lemma}[{\cite[Proposition~2.5.7]{Vieira_07}}] \label{lem:automorphism}
Let $(\bbV,\circ)$ be a Euclidean  Jordan algebra  with cone of squares $\calK$. 
  Then we have 
\begin{enumerate}[label=(\roman*)]
\item  \label{item:auto_inv} for any invertible $x\in \bbV$, $P(x)$ 
is an automorphism on both $\calK$ and 
$\inter\calK$,
\item \label{item:auto_pd} for any $x\in \inter \calK$, $P(x):\bbV\to\bbV$ is self-adjoint and positive definite.
\end{enumerate} 
\end{lemma}

\begin{lemma}[{\cite[Sections~2.3 and 2.5]{Vieira_07}}] \label{lem:calculus_P}
Consider the setting in Lemma~\ref{lem:automorphism}.  
Then we have 
\begin{enumerate}[label=(\roman*)]
\item $P(x)$ 
is invertible  if and only if $x\in\bbV$ is invertible, 
in which case 
$P(x)^{-1} = P(x^{-1})$, 
\item if both $x,y\in\bbV$ are invertible, then $P(x)y$ is invertible and $(P(x)y)^{-1} = P(x^{-1}) y^{-1}$,  
\item \label{item:fundamental} for all $x,y\in\bbV$, we have $P(P(x)y) = P(x)P(y)P(x)$.
\item for all $x\in\inter \calK(\bbV)$, $P(x^\alpha)x^\beta = x^{2\alpha+\beta}$ for any $\alpha,\beta\in \bbR$,  
\item for all $x\in\inter \calK(\bbV)$, $P(x^\alpha)=P(x)^\alpha$ for any  $\alpha\in \bbR$. 
\end{enumerate} 
\end{lemma}

\noindent 


\noindent
In proving Theorem~\ref{thm:barf}, we also need the following two lemmas. 

\begin{lemma}\label{lem:diff_ln}
Let $x\in\inter\calK$. 
Then for any $d\in\bbV$, 
\begin{align}
D\ln(x)[d] &=\medint\int_{0}^{+\infty} \, P(x+\tau e)^{-1} d \;\;\rmd \tau \label{eq:p'}\\
D^2 \ln(x)[d,d] &= -2  \medint\int_{0}^{+\infty} \,  P(x+\tau e)^{-1}P(d)(x+\tau e)^{-1} \;\;\rmd \tau. \label{eq:p''}
\end{align}
\end{lemma}

\begin{proof}
First note that the univariate function $\ln:
\bbR_{++}\to\bbR$ can be written as 
\begin{equation*}
\ln(\lambda) 
= \medint\int_{0}^{+\infty} \, ({\tau+1})^{-1} - ({\tau+\lambda})^{-1} \, \rmd \tau. 
\end{equation*}
For $x \in\inter\calK$, write its spectral decomposition as $x=\sum_{i=1}^n \sigma_i q_i$ (where $\sigma_i>0$ for $i\in[n]$), and 
\begin{align}
&\ln(x) =\textstyle \sum_{i=1}^n \ln (\sigma_i) q_i = \textstyle \sum_{i=1}^n \Big\{\medint\int_{0}^{+\infty} \, {(\tau+1)^{-1}} - (\tau+\sigma_i)^{-1} \; \rmd \tau \Big\}\, q_i\nn\\
 &\quad  =  \medint\int_{0}^{+\infty} \, ({\tau+1})^{-1} e - \textstyle \sum_{i=1}^n(\tau+\sigma_i)^{-1} q_i\, \rmd \tau 
  = \medint\int_{0}^{+\infty} \, ({\tau+1})^{-1} e  -(x+\tau e)^{-1} \; \rmd \tau. \label{eq:integral_ln(x)}
\end{align}
Now, define $\xi_x(\alpha):= (x+\alpha d)^{-1}$ with $\dom \xi_x:=\{\alpha\in\bbR: x+\alpha d\in\inter\calK\}$. Let $v:= P(x^{-1/2}) d$ 
with spectral decomposition $v=\sum_{i=1}^n \mu_i p_i$. For any $\alpha\in \dom \xi_x$, we have 
\begin{align*}
\xi_x(\alpha) = (P(x^{1/2})(e+\alpha v))^{-1} \eqa P(x^{-1/2})(e+\alpha v)^{-1} = \textstyle P(x^{-1/2}) \sum_{i=1}^n (1+\alpha \mu_i)^{-1} p_i, 
\end{align*}
where in (a) we use $e+\alpha v\in \inter\calK$, which follows from Lemma~\ref{lem:automorphism}\ref{item:auto_inv}. 
As a result, 
\begin{align}
 &
 \xi_x'(0)= -P(x^{-1/2}) \textstyle \sum_{i=1}^n (1+\alpha \mu_i)^{-2}\mu_i p_i\big\vert_{\alpha=0}
= -P(x^{-1/2}) v = -P(x)^{-1}d, \label{eq:Dx^-1} \\
\textstyle &
\xi_x''(0) = 2P(x^{-1/2}) \textstyle \sum_{i=1}^n \mu_i^2 p_i=2P(x^{-1/2})v^2\nn \\
&\qquad=  2P(x^{-1/2})(P(x^{-1/2}) d)^2 = 2P(x^{-1/2})P(P(x^{-1/2}) d)e = 2P(x)^{-1}P( d) x^{-1},\label{eq:D2x^-1}
\end{align}
where the last equality follows from Lemma~\ref{lem:calculus_P}\ref{item:fundamental}. 
Now, define $\zeta_x(\alpha):= \ln(x+\alpha d)$ with $\dom \zeta_x := \dom \xi_x$. 
From~\eqref{eq:integral_ln(x)}, we have 
\begin{equation}
\zeta_x(\alpha)=\medint\int_{0}^{+\infty} \, (\tau+1)^{-1} e  -(x+\tau e+\alpha d)^{-1}\, \rmd \tau = \medint\int_{0}^{+\infty} \, (\tau+1)^{-1} e  -\xi_{x+\tau e}(\alpha)\, \rmd \tau. 
\end{equation}
By the Leibniz integral rule,~\eqref{eq:Dx^-1} and~\eqref{eq:D2x^-1}, 
we have  
\begin{align*}
\textstyle D\ln(x)[d] = \zeta_x'(0) &= \medint\int_{0}^{+\infty} \,  -\xi_{x+\tau e}'(0)\, \rmd \tau = \medint\int_{0}^{+\infty} \,  P(x+\tau e)^{-1}d\, \rmd \tau\\
D^2 \ln(x)[d,d]  = \zeta_x''(0) &= \medint\int_{0}^{+\infty} \,  -\xi_{x+\tau e}''(0)\, \rmd \tau 
= -2\medint\int_{0}^{+\infty} \,  P(x+\tau e)^{-1}P(d)(x+\tau e)^{-1}\, \rmd \tau. \tag*{\qed}
\end{align*}
\renewcommand{\qedsymbol}{}
\end{proof}

\vspace{-2.5em}



\begin{lemma} \label{lem:inverse_partial_order}
Consider the setting in Lemma~\ref{lem:automorphism}, and let 
$\succeq\,:= \,\succeq_{\calK}$ and $\succ\,:= \,\succ_{\calK}$.  %
Let $a,b\in\bbV$. For all invertible $x\in\bbV$, we have $a\succeq  b$ if and only if $P(x)a\succeq P(x) b$. Also, if $a\succeq b\succ 0$, then 
$b^{-1}\succeq a^{-1}\succ 0$. 
\end{lemma}
\begin{proof}
See Appendix~\ref{app:proof_func}. 
\end{proof}

\noindent 
Equipped with all the lemmas above, we are ready to prove Theorem~\ref{thm:barf}.\\[-1ex]

\noindent
{\bf Proof of Theorem~\ref{thm:barf}.}
Let $\succeq\,:= \,\succeq_{\calK}$ and $\succ\,:= \,\succ_{\calK}$.  
Fix any $y\in\bbR^m_{++}$ and $h\in\bbR^m$, 
define 
\begin{equation}
\pi_y(\alpha):= \rvA^*\nabla \psi(y+\alpha h) = \textstyle \sum_{j=1}^m \nabla_j \psi(y+\alpha h)v_j, \;\; 
\dom \pi_y := \{\alpha\in\bbR:y+\alpha h \in\bbR_{++}^m\}.  \label{eq:pi_y} 
\end{equation}
Since $\psi$ is three-times differentiable on $\bbR_{++}^m$, 
$\pi_y$ is twice differentiable on $\dom \pi_y $, and 
\begin{equation}
\pi_y'(0) = \textstyle \sum_{j=1}^m D^2 \psi(y)[e_j,h]\;v_j\;\; \mbox{and}\;\;\; 
\pi_y''(0) = \textstyle \sum_{j=1}^m D^3 \psi(y)[e_j,h,h]\;v_j.\label{eq:pi_y'}
\end{equation}
Since $\nabla \psi: \bbR_{++}^m\to \bbR_{++}^m$,   
we have $\nabla \psi(y+\alpha h)\in \bbR_{++}^m$ for all $\alpha\in \dom \pi_y $. 
In addition, since 
$\rvA^*(\bbR_{++}^m) \subseteq \inter\calK$ (cf.\ Proposition~\ref{prop:assumption_A}), by~\eqref{eq:pi_y}, we have  
$\pi_y(\alpha) 
\in \inter\calK$ for all $\alpha\in \dom \pi_y $, and in particular, 
$\pi_y(0) \in \inter\calK$. 
Thus,  for all $\tau\ge 0$, we have $\pi_y(0)+\tau e\succeq \pi_y(0)\succ 0$, and by Lemma~\ref{lem:inverse_partial_order}, we have  $(\pi_y(0)+\tau e)^{-1}\preceq \pi_y(0)^{-1}$. Then by Lemma~\ref{lem:P(y)}, we have $P(\pi_y'(0))(\pi_y(0)+\tau e)^{-1} \preceq P(\pi_y'(0))\pi_y(0)^{-1}$, and hence 
\begin{equation}
(1/2)\pi_y''(0)   -   P(\pi_y'(0))(\pi_y(0)+\tau e)^{-1}\succeq (1/2)\pi_y''(0)   -   P(\pi_y'(0))\pi_y(0)^{-1},\quad \forall\, \tau\ge 0.  \label{eq:tau_to_0}
\end{equation}
Note that by the definitions of $\pi_y(0)$, $\pi'_y(0)$ and $\pi''_y(0)$ in~\eqref{eq:pi_y} and~\eqref{eq:pi_y'}, the condition in~\eqref{eq:cond_psi}, and the Cauchy–Schwarz inequality in Corollary~\ref{cor:CS}, we have 
\begin{align}
&(1/2)\pi_y''(0)   -   P(\pi_y'(0))\pi_y(0)^{-1}\nn\\
  \succeq\;  &  \sum_{j=1}^m \frac{(D^2 \psi(y)[e_j,h])^2}{D \psi(y)[e_j]} \;v_j   -   P\big(\textstyle  \sum_{j=1}^m D^2 \psi(y)[e_j,h]\;v_j\big)\big(\sum_{j=1}^m D \psi(y)[e_j] \; v_j\big)^{-1}  \succeq 0.  \label{eq:succeq_0}
\end{align}
Combining~\eqref{eq:tau_to_0} and~\eqref{eq:succeq_0}, we have  
\begin{equation}
(1/2)\pi_y''(0)   -   P(\pi_y'(0))(\pi_y(0)+\tau e)^{-1}\succeq 0,\quad \forall\, \tau\ge 0. \label{eq:tau_succeq_0}
\end{equation}
Now, since $\pi_y(\alpha) \in \inter\calK$ for all $\alpha\in \dom \pi_y $, let us define
\begin{equation}
\omega_y(\alpha):= \ln\pi_y(\alpha) =g(y+\alpha h), \quad {\rm where}\quad \dom \omega_y  := \dom \pi_y. 
\end{equation}
Note that for $\alpha\in \dom \omega_y$, we have $\omega_y'(\alpha)=D \ln(\pi_y(\alpha))[\pi_y'(\alpha)]$ and hence 
\begin{align}
\omega_y''(0) &= D \ln(\pi_y(0))[\pi_y''(0)]+D^2 \ln(\pi_y(0))[\pi_y'(0),\pi_y'(0)]. \label{eq:omega''}
\end{align}
By~\eqref{eq:p'} and~\eqref{eq:p''},  we then have 
\begin{align}
\omega_y''(0) 
&= \medint\int_{0}^{+\infty} \, P(\pi_y(0)+\tau e)^{-1} \pi_y''(0) \,\rmd \tau  -2  \medint\int_{0}^{+\infty} \,  P(\pi_y(0)+\tau e)^{-1}P(\pi_y'(0))(\pi_y(0)+\tau e)^{-1} \,\rmd \tau\nn\\
&= 2\medint\int_{0}^{+\infty} \, P(\pi_y(0)+\tau e)^{-1} \Big((1/2)\pi_y''(0)   -   P(\pi_y'(0))(\pi_y(0)+\tau e)^{-1}\Big) \,\rmd \tau\succeq 0,
\end{align}
where the inequality follows from~\eqref{eq:tau_succeq_0} and that $P(\pi_y(0)+\tau e)\in \Aut(\calK)$ (cf.~Lemma~\ref{lem:automorphism}\ref{item:auto_inv}). 
 Since $D^2 g(y)[h,h] = \omega_y''(0)$, we complete the proof.  \qed

\section{Comparison of Computational Complexity With Related FOMs} \label{sec:comp_complexity}

%

In this section, we compare the computational complexity of the {\tt GMG} method in Algorithm~\ref{algo:GMG}  with three other related FOMs on the four problems instances introduced in Section~\ref{sec:intro}, namely~\eqref{eq:PET}, \eqref{eq:DOPT},~\eqref{eq:QST} and~\eqref{eq:DBQP}.  Those three FOMs are the Barrier Subgradient Method ({\tt BSM})~\cite{Nest_11}, Relatively Smooth Gradient Method ({\tt RSGM})~\cite{Bauschke_17,Lu_18} and the Frank-Wolfe method for minimizing logarithmically homogeneous barriers ({\tt FW-LHB})~\cite{Zhao_23},  all of which were developed to solve certain classes of convex optimization problems without the Lipschitz-gradient condition. 

In the following, we will first briefly introduce {\tt BSM}, {\tt RSGM} and {\tt FW-LHB} based on solving the concave maximization problem $\max\{F(x): x\in\calC\}$, where $\calC$ is the same constraint set as in~\eqref{eq:P}, but the assumptions on $F$ vary from one method to another. Then, for each of the four problems instances in Section~\ref{sec:intro}, we will analyze and compare the computational complexities of all of the four FOMs mentioned above (including {\tt GMG}).   
As a remark, 
 the convergence rate of any method in this section is 
measured in terms of the objective gap. 


\subsection{Overview of Related FOMs}

\noindent 
{\bf BSM.}  For this method, the objective function $F: \inter\calK\to \bbR$ has the form $F := \ln\circ\, \psi$, where $\psi$ is concave and positive on $\inter\calK$. 
In particular, $\psi$ is ``super-differentiable'' on $\inter\calK$, namely, for all $x\in\inter\calK$, $\bar\partial f(x):= \{g\in\bbV: f(y)\le f(x)+\ipt{g}{y-x},\; \forall\, y\in\bbV\}\ne \emptyset$.  The \texttt{BSM} has the same structure as the original dual averaging method (cf.~\cite{Nest_09}), but with an ``atypical''
 prox-function $
 -\ln\det(\cdot): \inter\calK\to \bbR$, namely 
 the ``log-determinant barrier'' for the symmetric cone $\calK$. 
 (In fact, 
 $-\ln\det(\cdot)$ is ``atypical'' because it is not continuous on $\calC$, which is typically assumed for the  prox-function; see e.g.,~\cite{Nest_05}.)  
In the \texttt{BSM}, we choose the step-sizes $\{\beta_k\}_{k\ge 0}$ such that $\beta_0 := \beta_1$ and $\beta_k:= 1+\sqrt{k/n}$ for $k\ge 1$, and  set $s^0:= 0$.  At each iteration $k\ge 0$, the \texttt{BSM} iterates:
\begin{equation}
x^k: = {\argmax}_{x\in \calC}\; \ipt{s^k}{x} + \beta_k \ln\det(x), \qquad s^{k+1}:= s^k +  g_k,\quad  \where \;\; g^k\in \bar\partial f(x^k).
\end{equation}
As shown the proof of~\cite[Theorem~3]{Nest_11}, the \texttt{BSM} has a convergence rate of $O(\sqrt{n}\ln(nk)/\sqrt{k})$. Therefore, to obtain an $\varepsilon$-optimal solution of~\eqref{eq:P}, it requires $O((n/\varepsilon^2)\ln^2(n/\varepsilon))$ iterations. 

As discussed in~\cite[Section~7]{Nest_11}, to solve the projection sub-problem that generates $x^k$, we first perform an eigendecomposition of $s^k$ to obtain its vector of eigenvalues $\lambda^k\in\bbR^n$, and 
then  solve the following  projection problem:
\begin{equation}
{\max}_{\mu \in \Delta_n}\; \ipt{\lambda^k}{\mu} + \textstyle \sum_{i=1}^n \ln(\mu_i).  \label{eq:log_proj}
\end{equation}
The dual of this problem is a one-dimensional 
 self-concordant minimization problem, which can be solved by Newton's method in a standard way. 
 Nesterov showed that it takes $O(\ln n)$ iterations of Newton's method to arrive at the region of quadratic convergence, and 
each iteration of Newton's method takes $O(n)$ arithmetic operations. 
\\[-1ex]

\noindent 
{\bf RSGM.} To introduce this method, we first introduce the notion of {relative smoothness}. Let $h:\bbV\to\bbR\cup\{+\infty\}$ be a Legendre function (cf.~Definition~\ref{def:legendre}) such that $\calC\subseteq \cl\dom h$, $\dom h\subseteq \dom \barF$ (
cf.~\eqref{eq:def_-F})
 and $F$ is differentiable on $\inter\dom h$. We say that $f$ is $L$-relatively-smooth  w.r.t.\ $h$ (on $\inter\dom h$) if {there exists $L>0$ such that $Lh - f$  is convex on $
\inter\dom h$. 
If we can find such a ``reference function'' $h$, then the \texttt{RSGM} starts with $x^0\in \inter\dom h$, and iterates: 
\begin{equation}
x ^ {k+1} =  {\argmax}_{x \in \calC}\;\;  \langle \nabla F (x ^ {k }), x \rangle - L D_{h} (x, x ^ {k }), \quad \forall\,k\ge 0,  \label{eq:rsgm}
\end{equation}
where the Bregman divergence $D_h (y, x ):= h (y) - h (x) - \langle \nabla h (x), y - x \rangle$ for $y\in \dom h$ and $x\in \inter\dom h$. 
As shown in~\cite[Theorem~1]{Bauschke_17}, we have the following convergence rate:   
\begin{equation}
F(x) - F(x^k) 
\le {LD_h(x,x^0)}/{k}, \quad \forall\,k\ge 1, \;\; \forall\, x\in \dom h. \label{eq:rate_rsgm} 
\end{equation}
Let us remark a  subtle but important issue about the convergence rate in~\eqref{eq:rate_rsgm}. 
Since we only require $\calC \subseteq \cl\dom h$, the set of optimal solutions of~\eqref{eq:P}, denoted by $\calX^*\subseteq \calC$, 
may have empty intersection with $\dom h$. In this case, the  convergence rate in~\eqref{eq:rate_rsgm} does not directly translate to the convergence rate of the objective gap. A common approach to resolve this issue is to fix any accuracy $\varepsilon>0$ and $x^*\in\calX^*$, and    construct a surrogate point $\hatx_\varepsilon\in\inter\dom h$ 
as a  strict convex combination of $x^0\in  \inter\dom h$ and $x^*\in \cl\dom h$, such that the objective gap of $\hatx_\varepsilon$ is $\varepsilon/2$. Then by~\eqref{eq:rate_rsgm}, the number of iterations of \texttt{RSGM} to reach an $\varepsilon$-optimal solution can be bounded by $\lceil 2LD_h(\hatx_\varepsilon,x^0)/\varepsilon\rceil$. It then remains to bound the Bregman divergence $D_h(\hatx_\varepsilon,x^0)$. This requires exploiting the specific structures of $h$ and $x^0$, and 
so far has only been done case-by-case. 
Moreover, since $D_h(\hatx_\varepsilon,x^0)$ depends on $1/\varepsilon$ in general, the iteration complexity of \texttt{RSGM} is typically (slightly) worse than $O(1/\varepsilon)$. For examples of such an approach applied to~\eqref{eq:PET} and~\eqref{eq:DOPT}, we refer readers to~\cite[Proposition~2]{Zhao_23pet} and~\cite[Theorem~4.1]{Lu_18}.


In addition, note that  the computational complexity of the projection sub-problem in~\eqref{eq:rsgm} depends on both the constraint set $\calC$ and the reference function $h$. For~\eqref{eq:PET} and~\eqref{eq:DOPT} with rank-one matrices $\{A_i\}_{i=1}^n$, it was shown in~\cite{Bauschke_17} and~\cite{Lu_18} that the objective functions in both problems are relatively smooth to the ``log-barrier'' $h(x):= - \sum_{i=1}^n \ln(x_i)$ for $x\in\bbR_{++}^n$, and hence the projection sub-problem in~\eqref{eq:rsgm} 
becomes~\eqref{eq:log_proj}. As mentioned above, the dual of~\eqref{eq:log_proj} can be solved in $O(n\ln n)$ arithmetic operations via Newton's method.  
\\[-1ex]

\noindent 
{\bf FW-LHB.} 
This method can be regarded as a variant of the classical FW method~\cite{Frank_56,Jaggi_13,Freund_16} for solving convex optimization problems without the Lipschitz-gradient condition. 
In this method, $F$ has the same composite structure as in~\eqref{eq:P}, namely, $F = f\circ \rvA$, but with different assumptions on $f$ and $\rvA$. Specifically, $-f:\bbY\to \barbbR$ is a non-degenerate $\theta$-LHSCB on a regular cone $\Gamma$ for some $\theta\ge 1$ (cf.\ Remark~\ref{rmk:LHSCB} and Proposition~\ref{prop:equiv_barf}), and the linear operator $\rvA:\bbV\to \bbY$ 
is not required to satisfy Assumption~\ref{assump:A}.   The \texttt{FW-LHB} has the same structure as the classical FW method~\cite{Frank_56}, but with a very different choice of step-size. Specifically, starting with $x^0\in\ri\calC$, each iteration $k\ge 0$ reads: 
\begin{equation}
v^k\in{\argmax}_{x\in\calC}\; \lranglet{\nabla F( x^k)}{x}, \quad x^{k+1}:= x^k + \alpha_k (v^k-x^k), \label{eq:FW-LHB}
\end{equation}
where the step-size $\alpha_k:= \min\{{G_k}/({D_k(G_k+D_k)}) , 1\}$,  $D_k:=  \lranglet{\nabla^2 F(x^k)(v^k-x^k) }{v^k-x^k}^{1/2} $  and $G_k:= \lranglet{\nabla F(x^k)}{v^k - x^k}$. In~\cite[Theorem~2.1]{Zhao_23}, it was shown that to obtain an $\varepsilon$-optimal solution to~\eqref{eq:P}, the \texttt{FW-LHB} requires essentially $O(\theta^2/\varepsilon)$ iterations.

\subsection{Comparison of Computational Complexity }

In the following, for each of the four problem instances in Section~\ref{sec:intro}, we analyze the computational complexities of the four FOMs (namely {\tt BSM}, {\tt RSGM}, {\tt FW-LHB} and {\tt GMG}) whenever they are applicable, and summarize the results in Table~\ref{tab:complexity}. For ease of comparison, we let all the four FOMs share the starting same point $x^0 = (1/n) e$ on all the problem instances, 
and choose $\alpha=1$ in  {\tt GMG}. 
Also, we shall make and state clearly  some mild assumptions on the problem data of~\eqref{eq:PET}, \eqref{eq:DOPT} and~\eqref{eq:QST}. 
In addition, in analyzing the computational complexities of the four FOMs on these three problems, 
we will always work under the regime that $n\lesssim e^m$ (i.e., $n\le C e^m$ for some absolute constant $C>0$), which is 
the typical 
 case in various applications. 
\\[-1ex]

\noindent 
\underline{\eqref{eq:PET} with $p_1=\cdots=p_m = 1/m$}: For all of the four methods, the cost-dominating operation in each iteration is computing the gradient of $F$, which costs $O(mn)$ arithmetic operations. In addition, note that since $f(y):= (1/m)\textstyle\sum_{j=1}^m \ln y_j$ is not a LHSCB  on $\bbR_+^m$ (cf.\ Remark~\ref{rmk:LHSCB}), we apply  {\tt FW-LHB} instead to~\eqref{eq:PET} with the scaled objective $F$, 
 namely $\min\{mF(x):x\in\Delta_n\}$. Similarly, for {\tt FW-LHB} to solve the other three problems below, we also scale 
 $F$ if needed (so that $f$ becomes a LHSCB). In addition, the computational complexity of {\tt RSGM} can be obtained from~\cite[Proposition~2]{Zhao_23pet}. Since we focus on the regime $n\lesssim e^m$, it is clear from  Table~\ref{tab:complexity} that {\tt GMG} has the lowest computational complexity  among all the four methods. 
 \\[-1ex]
 
 \noindent 
\underline{\eqref{eq:DOPT} with rank-one matrices $\{A_i\}_{i=1}^n  \subseteq \bbS_{+}^m$}: 
First, note that since $\sum_{i=1}^n A_i \succ 0$, we must have $n\ge m$. 
This implies that for all of the four methods, the gradient computation is the cost-dominating operation in each iteration. 
Note that {\tt FW-LHB} computes the gradient of $F$ more efficiently compared to the other three methods. Specifically, by using the rank-one update, 
each iteration of {\tt FW-LHB} 
 computes the gradient 
 in $O(mn)$ arithmetic operations (cf.~\cite[Theorem~1]{Khachiyan_96}). In contrast, for the other three methods, the  gradient computation takes $O(m^2n)$ arithmetic operations per iteration. In addition, the computational complexity of {\tt RSGM} can be obtained from~\cite[Table~1]{Lu_18}. From  Table~\ref{tab:complexity}, it is clear that {\tt FM-LHB} has the lowest computational complexity  among all the four methods, while {\tt GMG} is slightly worse by a factor of $\ln(n)$.  
  \\[-1ex]
  
 \noindent 
\underline{\eqref{eq:QST} with $p_1=\cdots=p_m = 1/m$ and rank-one matrices $\{A_j\}_{j=1}^m\subseteq \bbH_+^n\setminus\{0\}$}: Unlike the two problems above, it is unclear if {\tt RSGM} can be applied to this problem, and if so, what its computational complexity should be.   Therefore, we only compare the other three methods.
First, note that since $\sum_{j=1}^m A_j \succ 0$, we must have $m\ge n$. Also, note that the gradient computation for all of the three methods costs $O(mn^2)$ arithmetic operations per iteration. Next, note that for {\tt FM-LHB}, the maximization sub-problem in~\eqref{eq:FW-LHB} amounts to finding the maximum eigenvalue of some $S\in \bbH^n$, 
which takes $O(n^2)$ arithmetic operations. In contrast,  both {\tt GMG} and {\tt BSM} involves the eigen-decomposition of some $S\in \bbH^n$ in each iteration, which takes $O(n^3)$ arithmetic operations. In either case, the gradient computation is the dominating operation for all of the three methods. Since we focus on the regime $n \lesssim e^m$, it is clear from  Table~\ref{tab:complexity} that {\tt GMG} has the lowest computational complexity  among all the three methods. 
 \\[-1ex]

 \noindent 
\underline{\eqref{eq:DBQP}}: As mentioned in Section~\ref{sec:another_app}, 
due to the ``non-standard'' behavior of $f$ on $\bdry \dom \barf$, 
this is a genuinely difficult problem to be solved by FOMs. In fact, only {\tt GMG} and {\tt BSM} have provable computational guarantees on this problem. For both methods, the gradient computation costs $O(n^3)$ arithmetic operations. In addition, both methods involve  the eigen-decomposition of some $S\in \bbS^n$ in each iteration, which also costs $O(n^3)$ arithmetic operations. 
 From Table~\ref{tab:complexity}, it is clear that {\tt GMG} has a much better computational complexity compared to {\tt BSM}, in terms of the dependence on both $n$ and $1/\varepsilon$. 

%

%
%





\begin{table}\centering\setlength{\tabcolsep}{8pt}\renewcommand{\arraystretch}{2}
\caption{Comparison of the computational complexities of {\tt BSM}, {\tt RSGM}, {\tt FW-LHB} and {\tt GMG} on the four problem instances in Section~\ref{sec:intro}. The question  mark ``?'' indicates that it is unclear if a FOM can be applied to a particular problem with provable computational guarantees. For each problem, the best computational complexity is highlighted in blue (under the regime that $n\lesssim e^m$). 
\label{tab:complexity}}
\vspace{-2ex}
\begin{tabular}{|c|c|c|c|c|}  
\hline
&\texttt{RSGM} & \texttt{FW-LHB} & \texttt{GMG} & \texttt{BSM} \\\hline
\ref{eq:PET} & $O\big(\frac{mn^2}{\varepsilon}\ln\big(\frac{\ln(n)}{\varepsilon}\big)\big)$ & $O\big(\frac{m^2 n}{\varepsilon}\big)$ & \blue{$O\big(\frac{m n\ln(n)}{\varepsilon}\big)$} & $O\big(\frac{m n^2}{\varepsilon^2}\ln^2\big(\frac{n}{\varepsilon}\big)\big)$ \\\hline
\ref{eq:DOPT} & $O\big(\frac{mn^2}{\varepsilon}\ln\big(\frac{\ln(n)}{\varepsilon}\big)\big)$ & \blue{$O\big(\frac{m^2 n}{\varepsilon}\big)$} & $O\big(\frac{m^2 n\ln(n)}{\varepsilon}\big)$ & $O\big(\frac{m^2 n^2}{\varepsilon^2}\ln^2\big(\frac{n}{\varepsilon}\big)\big)$  \\\hline
\ref{eq:QST} & ? & $O\big(\frac{m^2 n^2}{\varepsilon}\big)$ & \blue{$O\big(\frac{m n^2\ln(n)}{\varepsilon}\big)$} & $O\big(\frac{m n^3}{\varepsilon^2}\ln^2\big(\frac{n}{\varepsilon}\big)\big)$  \\\hline
\ref{eq:DBQP} & ? & ? & \blue{$O\big(\frac{n^3\ln(n)}{\varepsilon}\big)$} &  $O\big(\frac{n^4}{\varepsilon^2}\ln^2\big(\frac{n}{\varepsilon}\big)\big)$   \\\hline
\end{tabular}
\end{table}
%

\vspace{2ex}
\noindent 
{\bf Acknowledgment.} The author sincerely thanks Prof.\ Levent Tun\c{c}el, Prof.\ Bruno Louren\c{c}o,  and Prof.\ Robert Freund for helpful discussions during the preparation of this manuscript. The author's research is partially supported by AFOSR Grant No.\ FA9550-22-1-0356.



\appendix

\section{Proof of Lemma~\ref{lem:simplex} } \label{app:proof_simplex}

First, note that since $\Omega$ is a closed and convex cone, it is pointed if and only if $\Omega^*$ is solid. 
The following proof leverages results in~\cite[Chapters 6 \& 8]{Rock_70}.
Define $\calH_s:= \{x\in\bbV:\ipt{s}{x}=1\}$, so that $\calQ_s = \calA \cap \calH_s$ 
is a convex set.  Since $\ri\calA\ne\emptyset$, let $z\in \ri\calA$, and note that  $\ri\calA =\ri(\cl\calA)= \ri \Omega$. Since $\{0\}\subsetneq \Omega$ and $\Omega$ is pointed, $\Omega$ is not a linear subspace, and hence $0\not\in\ri\Omega$. 
Since $z\in \ri \Omega$,  we know that $z\ne 0$, and hence $\ipt{s}{z}>0$. As a result, $z':= z/\ipt{s}{z}\in(\ri \Omega)\cap \calH_s = (\ri \calA)\cap \calH_s$. Consequently, $\ri \calQ_s = \ri(\calA\cap \calH_s) = (\ri\calA)\cap \calH_s=(\ri\Omega)\cap \calH_s\ne \emptyset$ and $\calC_s:=\cl \calQ_s =  (\cl\calA)\cap \calH_s =\Omega\cap \calH_s. $ 
Since $\ri \calQ_s\ne \emptyset$, both $\calQ_s$ and $\cl \calQ_s$ are nonempty. In addition, since $\calQ_s$ is convex, so are $\ri\calQ_s$ and $\cl \calQ_s$. 
Now, it remains to show that $\calC_s$ is bounded, which implies that both $\ri \calQ_s$ and $\calQ_s$ are bounded.  
Let $\R_{\calC_s}$ denote the recession cone of $\calC_s$, and 
we have $\R_{\calC_s} = \R_{\Omega}\cap\R_{\calH_s} = \{x\in \Omega:\ipt{s}{x}=0\}$. Since $s\in \inter\Omega^*$, we have $\R_{\calC_s} = \{0\}$. Since $\calC_s$ is closed and convex, this implies that $\calC_s$ is bounded. 
 \qed
 
\section{Proofs in Section~\ref{sec:algo} } \label{app:proof_analysis}

{\bf Proof of Lemma~\ref{lem:lambda_max}.}
Write the spectral decomposition $x = \sum_{i=1}^n \lambda_i(x) q_i$.  
Since $z\in\calC$ and 
$\{q_i\}_{i=1}^n\subseteq \calC$, we have $\ipt{z}{q_i}\ge 0$ for all $i\in[n]$ and $\sum_{i=1}^n \ipt{z}{q_i}=\ipt{e}{z}= 1$. Therefore,  
\begin{equation}
{\max}_{z\in\calC}\; \ipt{z}{x} = {\max}_{z\in\calC}\; \textstyle\sum_{i=1}^n \lambda_i(x)\, \ipt{z}{q_i} \le  {\max}_{\beta\in\Delta_n}\; \textstyle\sum_{i=1}^n \lambda_i(x)\, \beta_i = \lambda_{\max}(x).
\end{equation}
In addition,  since $q_1\in \calC$, 
we have ${\max}_{z\in\calC}\; \ipt{z}{x}\ge \ipt{q_1}{x} = \lambda_1(x)\normt{q_1}^2 = \lambda_{\max}(x)$.
For the second claim, note that since $q_1\in\calC$, if $x\preceq y$, then 
 $\lambda_{\max}(x) = \ipt{q_1}{x} \le \ipt{q_1}{y}\le \lambda_{\max}(y)$.
 \qed\\[-1ex]

\noindent
{\bf Proof of Lemma~\ref{lem:alpha}.} Write the spectral decomposition $y = \sum_{i=1}^n \lambda_i q_i$. Similar to the proof of Lemma~\ref{lem:lambda_max}, since $x\in\calC$, we have $\ipt{x}{q_i}\ge 0$ for all $i\in[n]$ and $\sum_{i=1}^n \ipt{x}{q_i}= 1$. Since $\alpha\in(0,1]$, $t\mapsto t^\alpha$ is concave on $\bbR_{+}$, and hence 
\begin{align*}
\textstyle
\ipt{x}{y^\alpha} = \sum_{i=1}^n \ipt{x}{q_i} \lambda_i^\alpha  \le \big(\sum_{i=1}^n \ipt{x}{q_i} \lambda_i\big)^\alpha= \ipt{x}{y}^\alpha. \tag*{$\qed$}
\end{align*}

\noindent
{\bf Proof of Lemma~\ref{lem:error_bd}. }
Define 
$\calY:= \rvA(\calC)$. From Proposition~\ref{prop:well_posed}, we know that $\dom \barf \cap \calY\ne \emptyset$, and $\calX^*\ne\emptyset$. Then by Corollary~\ref{cor:error_bd}, for any $x^*\in \calX^*\subseteq\calC$ and $x\in\inter\calK$, 
\begin{align*}
F^* - F(x) &= \barf(\rvA x) - \barf(\rvA x^*) \le \theta\ln\big({-\ipt{\nabla \barf(\rvA x)}{\rvA x^*}}/{\theta}\big)\\
\qquad &= \theta\ln\big({\ipt{\nabla F(x)}{x^*}}/{\theta}\big)
 \le \theta\ln\big({{\max}_{z\in\calC}\;\ipt{\nabla F(x)}{z}}/{\theta}\big) = \theta\ln\big({\lambda_{\max}(\nabla F(x))}/{\theta}\big). 
\end{align*}
To complete the proof, it remains to note that for any $z\in\inter\calK$, 
\begin{align*}
\theta\ln\big({\lambda_{\max}(z)}/{\theta}\big) = \theta \, {\max}_{i\in[n]}\, \ln\big(\lambda_{i}(z)/{\theta}\big) = \theta \, {\max}_{i\in[n]}\, \lambda_{i}\big(\ln(z/{\theta})\big)=\theta\lambda_{\max}\big(\ln({z}/{\theta})\big). \tag*{$\qed$}
\end{align*}

\section{Proofs in Section~\ref{sec:function} } \label{app:proof_func}

{\bf Proof of Proposition~\ref{prop:LGrad_convex}.} 
It suffices to prove the following lemma.
\begin{lemma}
 $x\in \calK$ if and only if $\ipt{x}{u}\ge 0$ for all $u\in\PI(\bbV)$. 
\end{lemma}

\begin{proof}
By the self-duality of $\calK$, 
$x\in\calK$ if and only if $\ipt{x}{z}\ge 0$ for all $z\in\calK$. Since $\PI(\bbV)\subseteq \calK$, 
we have $\ipt{x}{u}\ge 0$ for all $u\in\PI(\bbV)$. Conversely, write the spectral decomposition $x=\sum_{i=1}^n \mu_i q_i$. 
Since  $q_i\in \PI(\bbV)$ for $i\in[n]$,   we have 
$\mu_i = \mu_i \normt{q_i}^2 = \ipt{x}{q_i}   \ge 0$ for $i\in[n]$, and thus $x\in\calK$. 
\end{proof}

\noindent
{\bf Proof of Proposition~\ref{prop:equiv_barf}.}
First,  since $\barf$ is essentially smooth, (c) and (d) are clearly equivalent. Next, we show that (a) $\Rightarrow$ (b) $\Rightarrow$ (c) $\Rightarrow$ (a). If $p$ is complete, then from~\eqref{eq:Z_p}, we know that  $\dom \barf = \calC(d)$ does not contain a line, and by~\cite[Theorem~4.1.3]{Nest_04}, we know that $\barf$ is non-degenerate, which implies that $\barf$ is strictly convex on $\calC(d)$, i.e., $\barf$ is essentially strictly convex. Now, if $p$ is not complete, then from~\eqref{eq:Z_p}, there exists $h\ne 0$ such that $h\in \calL(\calC(d))$, and hence for some $y\in \calC(d)$, the line $\calL_{y,h}:=\{y+th:t\in\bbR\}\subseteq \calC(d) $. 
By~\cite[Corollary~2.3.1]{Nest_94}, we know that 
$\barf$ is constant along this line, 
and therefore $\barf$ is not strictly convex on $\calC(d)$. \qed 

\vspace{1ex}

\noindent
{\bf Proof of Lemma~\ref{lem:sum_v_i}.}
It suffices to 
consider $\Omega\ne \bbV$, which implies that $\{0\}\subsetneq\Omega^*$. 
Since $\Omega$ is nonempty, closed and convex, we have $\Omega^{**} = \Omega$ and 
and hence 
\begin{equation}
\inter\Omega= \inter \Omega^{**} = \{x\in\bbV: \ipt{x}{z}> 0,\;\; \forall\,z\in\Omega^*\setminus\{0\} \}. \label{eq:inter_K}
\end{equation}
Fix any $z\in\Omega^*\setminus\{0\}$.  Since  $\sum_{i=1}^n v_i \in\inter\Omega$, 
by~\eqref{eq:inter_K}, we have 
$\sum_{i=1}^n \ipt{v_i}{z} >0$.  
Since $\{v_i\}_{i=1}^n\subseteq \Omega$,  this implies that there exists $j\in[n]$ such that $\ipt{v_j}{z}>0$. Therefore, for any $\alpha_i>0$, $i\in[n]$, we have $\ipt{\sum_{i=1}^n \alpha_iv_i}{z} =\sum_{i=1}^n \alpha_i\ipt{v_i}{z} >0$. Since this holds for any $z\in\Omega^*\setminus\{0\}$, by~\eqref{eq:inter_K}, 
we have $\sum_{i=1}^n \alpha_i v_i\in\inter\Omega$. \qed


\vspace{1ex}

\noindent
{\bf Proof of Lemma~\ref{lem:inverse_partial_order}.}
To prove the first part, it suffices to show that for all invertible $x\in\bbV$,  $a\in\calK$ if and only if $P(x)a\in\calK$, 
which 
directly follows from Lemma~\ref{lem:automorphism}\ref{item:auto_inv}. 
Next, note that if $a\succeq e$, then $e\succeq a^{-1}$. To see this, 
note that if $a\succeq e$, then $\lambda_i(a)\ge 1$ for $i\in[n]$. 
As a result, $1/\lambda_i(a)\le 1$ for $i\in[n]$, which leads to $e\succeq a^{-1}$. 
Now, if $a\succeq b \succ 0$, then $P(b^{-1/2})a \succeq P(b^{-1/2})b = e$, and hence $P(b^{1/2})a^{-1} 
\preceq e$. As a result, we have  $0\prec a^{-1}
\preceq P(b^{-1/2})e = b^{-1}$. \qed 



\scriptsize
\bibliographystyle{IEEEtr}
\bibliography{math_opt,mach_learn,GF-papers-nips-better,stat_ref}

\end{document}